\documentclass[10pt,journal]{IEEEtran}
\usepackage[utf8]{inputenc}
\usepackage{amssymb}
\usepackage{amsthm}
\usepackage{amsmath}
\usepackage{enumerate}
\usepackage{booktabs}
\usepackage{graphicx}
\usepackage{multirow}
\usepackage{authblk}
\usepackage{hyperref}
\usepackage{cite}

\usepackage{color}
\usepackage{ulem}

\newtheorem{theorem}{Theorem}

\newtheorem{lemma}{Lemma}
\newtheorem{definition}{Definition}

\newcommand{\norm}[1]{\left\lVert#1\right\rVert}
\newcommand{\brac}[1]{\left(#1\right)}
\newcommand{\R}{\mathbb{R}}
\newcommand{\C}{\mathbb{C}}

\ifCLASSINFOpdf
\else
\fi
\begin{document}
%
\title{Inverse modified differential equations for discovery of dynamics}


\author{\IEEEauthorblockN{Aiqing Zhu,
Pengzhan Jin,
Beibei Zhu, and
Yifa Tang}

\thanks{
This work is supported by the Major Project on New Generation of Artificial Intelligence from MOST of China (Grant No. 2018AAA0101002), and National Natural Science Foundation of China (Grant Nos. 11771438 and 11901564). Corresponding author: Yifa Tang (email: tyf@lsec.cc.ac.cn).

Aiqing and Yifa Tang are with LSEC, ICMSEC, Academy of Mathematics and Systems Science, Chinese Academy of Sciences, Beijing 100190, China and School of Mathematical Sciences, University of Chinese Academy of Sciences, Beijing 100049, China (email: zaq@lsec.cc.ac.cn; tyf@lsec.cc.ac.cn)

Pengzhan Jin is with School of Mathematical Sciences, Peking University, Beijing 100871, China. (email: jpz@math.pku.edu.cn)

Beibei Zhu is with School of Mathematics and Physics, University of Science and Technology Beijing, Beijing 100083, China (email: zhubeibei@lsec.cc.ac.cn)

}}

%




\maketitle

\IEEEdisplaynontitleabstractindextext

%
\IEEEpeerreviewmaketitle

\begin{abstract}
The combination of numerical integration and deep learning, i.e., ODE-net, has been successfully employed in a variety of applications. In this work, we introduce inverse modified differential equations (IMDE) to contribute to the behaviour and error analysis of discovery of dynamics using ODE-net. It is shown that the difference between the learned ODE and the truncated IMDE is bounded by the sum of learning loss and a discrepancy which can be made sub exponentially small. In addition, we deduce that the total error of ODE-net is bounded by the sum of discrete error and learning loss. Furthermore, with the help of IMDE, theoretical results on learning Hamiltonian system are derived. Several experiments are performed to numerically verify our theoretical results.
\end{abstract}


\bigskip

\section{Introduction}
%
%
%
%

Identification of nonlinear system is a significant task existing in diverse applications \cite{brunton2019data,schmidt2009distilling}. Neural network  has became a powerful approach for such task, and a series of continuous models combined numerical integrator and neural networks had already been developed and implemented to learn hidden dynamics decades ago \cite{anderson1996comparison,gonzalez1998identification,rico1994continuous,rico1993continuous}. Recently, neural network is experiencing a renaissance with the growth of available data and computing resources. At the same time, many researchers pay attention to the connection between dynamical systems and deep neural networks and have done many related works in terms of algorithms, architectures and applications \cite{chen2018neural,e2017proposal,e2019mean, li2017maximum}. In particular, the continuous models have again attracted more and more attention and several ODE based models have been developed for discovery of hidden dynamics \cite{bertalan2019learning,chen2018neural,kolter2019learning,raissi2018multistep,yu2020onsagernet}.

Numerical integration plays an essential role in these ODE based models, as reported in \cite{gusak2020towards,ott2021resnet,queiruga2020continuous,zhuang2020adaptive}. However, their rigorous analysis is still under investigation. In \cite{keller2021discovery}, a framework based on refined notions is established for convergence and stability analysis of Linear Multistep Neural Networks (LMNets) \cite{raissi2018multistep}. Error estimation is enriched in \cite{du2021discovery}, which indicates that the grid error of LMNets is bounded by the sum of discrete error and approximation error under auxiliary initial conditions. Different from their work, we concentrate on the behaviour and analysis of general ODE-net. Here, the unknown governing vector field is approximated by neural networks with given several phase points as training set. The training process is to minimize the difference between real states and predicted outputs of an ODE solver.

The main ingredient of this work is formal analysis \cite{feng1991formal}. Historically, modified differential equation is an important tool for understanding the numerical behavior of ordinary differential equation (ODE) \cite{eirola1993aspects,feng1993formal,sanz1992symplectic,yoshida1993recent}. The methodology is to interpret the numerical solution as the exact solution of a perturbed equation. In addition, modified integrator \cite{chartier2007numerical}
is developed for high order numerical integration. For a given system of differential equation, they tried to search a perturbed differential equation such that its numerical solution matches the exact solution of the original system. In this paper, we use the same idea as modified integrator but for analysis of discovery using ODE-net. It is shown that training ODE-net returns an approximation of the perturbed equation. We name the obtained perturbed equation as inverse modified differential equation (IMDE) since discovery is an inverse problem.

We first apply the IMDE for general ODE solver and prove that several compositions of a numerical integrator has the same IMDE as the numerical integrator itself. In addition, IMDE approach can be applied to LMNet, results in explicit recursion formula for linear multistep method. Furthermore, learning Hamiltonian system is also discussed. It is found that for a Hamiltonian system, the IMDE based on the symplectic integrator is still a Hamiltonian system.

The formal series expressing IMDE does not converge in general and has to be truncated. Following conventional truncation theory \cite{benettin1994hamiltonian,hairer1999backward,hairer1997life,reich1999backward}, the truncation inequalities are tailored to IMDE scenario under analyticity assumption. It is shown that the difference can be made sub exponentially small. As a result, we obtain the rigorous error analysis for discovery using ODE-net. In summary, we list several statements derived via IMDE that will be documented in detail later:
\begin{itemize}
    \item ODE-nets have almost certain approximation target, i.e., the difference between the learned vector field $f_{net}$ and the truncation of the vector field of IMDE $f_h^N$ is bounded by the sum of learning loss and a discrepancy which can be made sub exponentially small in the data step.
    \item The error between the trained network $f_{net}$ and the unknown vector field $f$ is bounded by the sum of discrete error $Ch^p$ and learning loss, where $h$ is the discrete step and $p$ is the order of the numerical integrator.
    \item Both ODE-net using non-symplectic integrators and LMNet tend not to learn conservation laws theoretically.
    \item HNN with symplectic integrator have almost certain approximation target. However, this conclusion is not always true for HNN with non-symplectic numerical integrators.
\end{itemize}

The rest of this paper is organized as follows. In Section \ref{sec:pre}, we briefly present some necessary notations, numerical integration, modified differential equations and modified integrator. The existing ODE based network architectures including ODE-net, LMNet and HNN are also introduced. In Section \ref{sec:IMDE}, we investigate IMDE for these learning models. In particular, learning Hamiltonian system is discussed. The rigorous analysis for ODE-net are detailed in Section \ref{sec:error}. In Section \ref{sec:Numerical results}, several numerical results are provided to verify the theoretical findings. Section \ref{sec:summary} contains a brief summary and several comments on the future work.

\section{Preliminaries}\label{sec:pre}
Without loss of generality, the attention in this paper will be addressed to autonomous systems of first-order ordinary differential equations
\begin{equation}\label{eq:ODE}
\frac{d}{dt}y(t) = f(y(t)),\quad y(0)=x,
\end{equation}
where $y(t) \in \mathbb{R}^N$ and $f:\mathbb{R}^N \rightarrow \mathbb{R}^N$ is smooth. The initial value is denoted as $x$ in this paper. A non-autonomous system $\frac{d}{dt}y(t) = f(t,y(t),p)$ with parameter $p$ can be brought into this form by appending the equation $\frac{d}{dt}t = 1$ and $\frac{d}{dt}p = 0$. Let $\phi_{t}(x)$ be the exact solution and $\Phi_{h}(x)$ be the numerical solution with discrete step $h$. In order to emphasize specific differential equation, we will add the subscript $f$ and denote $\phi_t$ as $\phi_{t,f}$ and $\Phi_h$ as $\Phi_{h,f}$ . The choice for ODE solver in this paper is $S$ compositions of a numerical integrator, i.e.,
\begin{equation*}
\text{ODESolve}(x, f, T)=\underbrace {\Phi_{h,f} \circ \cdots \circ \Phi_{h,f}}_{\text{ $S$ compositions}}(x)= \brac{\Phi_{h,f}}^S(x),
\end{equation*}
where $T=Sh$ with discrete step $h$ and composition number $S$.

\subsection{Numerical integration: brief review}
In the last few decades, several kinds of numerical integrations have been developed for ordinary differential equations, including Runge-Kutta methods and linear multistep methods. We recall some basic definitions and essential supporting results here. Refer to \cite{butcher1987the, hairer2006geometric, hairer1996solving} for more presentations of integrators. Below we first present the concepts of order and consistency.

\textbf{Order.}
An integrator $\Phi_{h}(x)$ with discrete step $h$ has order $p$, if for any sufficiently smooth equation (\ref{eq:ODE}) with arbitrary initial value $x$,
\begin{equation*}
\Phi_{h,f}(x)=\phi_{h,f}(x)+O(h^{p+1}).
\end{equation*}

\textbf{Consistency.}
An integrator is consistent if it has order $p\geq 1$.

\subsubsection{Runge-Kutta methods}
Let $b_i,a_{ij}\ (i,j=1,\cdots, s)$ be real numbers and let $c_i=\sum_{j=1}^sa_{ij}$. An $s$-stage Runge-Kutta method for (\ref{eq:ODE}) is defined as
\begin{equation}\label{runge-kutta}
\begin{aligned}
k_i =& f\left(y_0+h\sum_{j=1}^sa_{ij}k_j\right), \quad i=1,\cdots ,s,\\
y_1 =& y_0+h\sum_{i=1}^sb_ik_i,
\end{aligned}
\end{equation}
where the function $f$ is given and $\Phi_{h,f}(y_0)=y_1$. The method is explicit if $a_{ij}=0$ for $i\leq j$ and implicit otherwise. For sufficiently small $h$, the slopes $k_1,\cdots, k_s$ have local solutions close to $f(y_0)$ guaranteed by Implicit Function Theorem.

\begin{theorem}\label{the:RKexp}
The derivatives of the solution of a Runge-Kutta method (\ref{runge-kutta}) with respect to $y_0$, for $h=0$, are given by
\begin{equation*}
y_1^{(k)}|_{h=0}=\sum_{|\tau|=k}\gamma(\tau) \cdot \alpha(\tau) \cdot \phi(\tau) \cdot F(\tau)(y_0).
\end{equation*}
Here, $\tau$ is called trees and $|\tau|$ is the order of $\tau$ (the number of vertices). $\gamma(\tau)$, $\phi(\tau)$, $\alpha(\tau)$ are positive integer coefficients, $F(\tau)(y)$ is called elementary differentials and typically composed of $f(y)$ and its derivatives.
\end{theorem}

\begin{table}[htbp]
    \centering
    \begin{tabular}{|c|c|c|c|c|c|}
    \hline
       $|\tau|$& $\tau$             &$\gamma(\tau)$& $\alpha(\tau)$ & $\phi(\tau)$ & $F(\tau)$ \\
     \hline
     \hline
       1&$\bullet$                  &1 &1 & $\sum_ib_i$                        & $f$           \\
     \hline
       2&$[\bullet]$                &2 &1 & $\sum_{ij}b_ia_{ij} $              & $f'f$         \\
     \hline
       3&$[\bullet,\bullet]$        &3 &1 & $\sum_{ijk}b_ia_{ij}a_{ik}$        & $f''(f,f)$    \\

       3&$[[\bullet]]$              &6 &1 & $\sum_{ijk}b_ia_{ij}a_{jk}$        & $f'f'f$       \\
     \hline
       4&$[\bullet,\bullet,\bullet]$&4 &1 & $\sum_{ijkl}b_ia_{ij}a_{ik}a_{il}$ & $f'''(f,f,f)$ \\

       4&$[[\bullet],\bullet]$      &8 &3 & $\sum_{ijkl}b_ia_{ij}a_{ik}a_{jl}$ & $f''(f'f,f)$  \\

       4&$[[\bullet,\bullet ]]$     &12&1 & $\sum_{ijkl}b_ia_{ij}a_{jk}a_{jl}$ & $f'f''(f,f)$  \\

       4&$[[[\bullet]]]$            &24&1 & $\sum_{ijkl}b_ia_{ij}a_{jk}a_{kl}$ & $f'f'f'f$     \\
     \hline
    \end{tabular}
    \caption{Trees, elementary differentials and coefficients}
    \label{tab:RKtree}
\end{table}

Some $\gamma(\tau), \alpha(\tau), \phi(\tau), F(\tau)$ are reported in Table \ref{tab:RKtree}, detailed proof and computation can be found in \cite[Chapter \uppercase\expandafter{\romannumeral3}]{hairer2006geometric}. Here, the notation $f'(x)$ is a linear map (the Jacobian), the second order derivative $f''(x)$ is a symmetric bilinear map and similarly for higher order derivatives described as tensor. Due to Theorem \ref{the:RKexp}, the formal expansion of a Runge-Kutta method with initial condition $y_0=x$ is given by
\begin{equation*}
\Phi_{h,f}(x)=y+hd_{1,f}(x)+h^2d_{2,f}(x)+ \cdots,
\end{equation*}
where
\begin{equation*}
d_{k,f}(x) = \frac{1}{k!}y_1^{(k)}|_{h=0}=\frac{1}{k!}\sum_{|\tau|=k}\gamma(\tau) \cdot \alpha(\tau) \cdot \phi(\tau) \cdot F(\tau)(x).
\end{equation*}

\subsubsection{Linear multistep methods}
For first order differential equations (\ref{eq:ODE}), linear multistep methods are defined by the formula
\begin{equation}\label{eq:LMM}
\sum_{m=0}^M \alpha_my_{m} - h\sum_{m=0}^M\beta_mf(y_{m})=0,
\end{equation}
where $\alpha_m,\beta_m$ are real parameters, $\alpha_M \neq 0$  and $|\alpha_0|+ |\beta_0|>0$. In \cite{feng1998the}, it is shown that weakly stable multistep methods are essentially equivalent to one-step methods.
\begin{theorem}\label{the:step-transition operator}
Consider a weakly stable multistep method (\ref{eq:LMM}), there exists a unique formal expansion
\begin{equation*}
\Phi_{h,f}(x)=y+hd_{1,f}(x)+h^2d_{2,f}(x)+ \cdots
\end{equation*}
such that
\begin{equation*}
\sum_{m=0}^M \alpha_m\Phi_{mh,f}(x) = h\sum_{m=0}^M\beta_mf(\Phi_{mh,f}(x))
\end{equation*}
for arbitrary initial value $x$, where the identity is understood in the sense of the formal power series in $h$.
\end{theorem}
Here, weak stability requires
\begin{equation}\label{eq:weak stability}
\sum_{m=0}^M m\cdot \alpha_m \neq 0,
\end{equation}
$\Phi_{h,f}(x)$ is called ``step-transition operator'' \cite{feng1998the}, which also provides the formal expansion of linear multistep methods.

\subsubsection{Symplectic integration methods}
For even dimension $N$, denote the $N/2$-by-$N/2$ identity matrix by $I$, and let
\begin{equation*}
J=\begin{pmatrix} 0 & I \\ -I & 0\end{pmatrix}.
\end{equation*}
\begin{definition}
A differentiable map $g : U \rightarrow \mathbb{R}^{N}$ (where $N$ is even and $U\subseteq \mathbb{R}^{N}$ is an open set) is called symplectic if
\begin{equation*}
g'(x)^{T}Jg'(x)=J,
\end{equation*}
where $g'(x)$ is the Jacobian of $g(x)$.
\end{definition}
A Hamiltonian system is given by
\begin{equation}\label{eq:Hami}
\frac{d}{dt}y = J^{-1} \nabla H(y), \quad y(0)=x,
\end{equation}
where $y \in  \mathbb{R}^{N}$ and $H$ is the Hamiltonian function typically representing the energy of (\ref{eq:Hami}) \cite{arnold2013mathematical, arnold2007mathematical}. A remarkable property of Hamiltonian system is the symplecticity of the phase flow, which is proved by Poincar\'{e} in 1899 \cite[Section 38]{arnold2013mathematical}, i.e.,
\begin{equation*}
\phi_t'(x)^T J \phi_t'(x) = J,
\end{equation*}
where $\phi_t'(x) = \frac{\partial \phi_t(x)}{\partial x}$ is the Jacobian of $\phi_t$. Due to the intrinsic symplecticity, it is natural to search for numerical methods that preserve this structure, i.e., make $\Phi_h$ be a symplectic map. There are some well-developed works on symplectic integration, see for example \cite{feng1984on, feng1986difference, hairer2006geometric,sanz1992symplectic}. It should be noticed that both linear multistep method and explicit
Runge-Kutta method can not be always symplectic \cite{hairer2006geometric,tang1993symplecticity}.

\subsubsection{Lie derivatives}
Following \cite{hairer2006geometric}, we briefly review Lie derivatives. Given (\ref{eq:ODE}), Lie derivative $D$ is the differential operator defined as:
\begin{equation*}
Dg(y)=g'(y)f(y)
\end{equation*}
for $g:\mathbb{R}^N \rightarrow \mathbb{R}^M$. According to the chain rule, we have
\begin{equation*}
\frac{d}{dt}g(\phi_{t,f}(x))=(Dg)(\phi_{t,f}(x))
\end{equation*}
and thus obtain the Taylor series of $g(\phi_{t,f}(x))$ developed at $t=0$:
\begin{equation}\label{eq:Ld}
g(\phi_{t,f}(x))=\sum_{k=0}^{\infty}\frac{t^k}{k!}(D^kg)(x).
\end{equation}
In particular, by setting $t=h$ and $g(y)=I_N(y)=y$, the identity map, it turns to the Taylor series of the exact solution $\phi_{h,f}$ itself, i.e.,
\begin{equation}\label{eq:exasolu}
\begin{aligned}
\phi_{h,f}(x)=&\sum_{k=0}^{\infty}\frac{h^k}{k!}(D^kI_N)(x)\\
=&x+hf(x)+\frac{h^2}{2}f'f(x)\\
&+\frac{h^3}{6}(f''(f,f)(x)+f'f'f(x))+\cdots .
\end{aligned}
\end{equation}

\subsection{Modified differential equations and modified integrator}\label{mde}
Modified differential equation is a well-established tool for numerical treatment of ordinary differential equation. The approach is to search a perturbed differential equation
\begin{equation}\label{eq:mde}
\frac{d}{dt}\tilde{y}(t)=f_h(\tilde{y}(t))  = f_0(\tilde{y})+hf_1(\tilde{y})+h^2f_2(\tilde{y})+\cdots,
\end{equation}
such that $\Phi_{h,f}(x) = \phi_{h,f_h}(x)$ formally, where $\Phi_{h,f}(x)$ is the numerical solution of (\ref{eq:ODE}) and $\phi_{h,f_h}(x)$ is the exact solution of (\ref{eq:mde}). Expanding $\phi_{h,f_h}(x)$ and $\Phi_{h,f}$ into power series of $h$ and comparing equal powers yields recursion formulas for $f_k$. Refer to \cite[Section 9]{hairer2006geometric} for detailed computation.

Modified integrator is an approach for constructing high order methods via modified differential equations \cite{chartier2007numerical}. The idea is to find a perturbed differential equation
\begin{equation}\label{eq:imde}
\frac{d}{dt}\tilde{y}(t)=f_h(\tilde{y}(t))  = f_0(\tilde{y})+hf_1(\tilde{y})+h^2f_2(\tilde{y})+\cdots,
\end{equation}
such that $\Phi_{h,f_h}(x) = \phi_{h,f}(x)$ formally. Here, the identity is understood in the sense of the formal power series in $h$.

For implementation, we first expand the numerical solution,
\begin{equation}\label{eq:numsolu}
\Phi_{h,f_h}(x) = x + hd_{1,f_h}(x) + h^2d_{2,f_h}(x) + \cdots,
\end{equation}
where the functions $d_{j,f_h}$ are given and typically composed of $f_h$ and its derivatives. For consistent integrators,
\begin{equation*}
d_{1,f_h}(x) = f_h(x) = f_0(x)+hf_1(x)+h^2f_2(x)+\cdots.
\end{equation*}
In $h^id_{i,f_h}(x) $, the powers of $h$ of the terms containing $f_k$ is at least $k+i$. Thus the coefficients of $h^{k+1}$ in (\ref{eq:numsolu}) is
\begin{equation*}
f_k+ \cdots,
\end{equation*}
where the ``$\cdots$'' indicates residual terms composed of $f_j$ with $j\leq k-1$ and their derivatives. By comparison of the coefficients of  like powers of $h$ in (\ref{eq:exasolu}) and (\ref{eq:numsolu}), unique functions $f_k$ in (\ref{eq:imde}) are obtained recursively. In particular, for a method of order $p$, the functions $f_1, \cdots f_{p-1}$ vanish identically.
\begin{theorem}\label{the:modiode}
Suppose that the integrator $\Phi_{h}(x)$ with discrete step $h$ is of order $p\geq 1$, more precisely,
\begin{equation*}\label{eq:truncation}
\Phi_{h,f}(x)=\phi_{h,f}(x)+h^{p+1}\delta_f(x)+O(h^{p+2}),
\end{equation*}
where $h^{p+1}\delta_f(x)$ is the leading term of the local truncation applied to (\ref{eq:ODE}). Then, the IMDE obeys
\begin{equation*}
\frac{d}{dt}\tilde{y}=f_h(\tilde{y})=f(\tilde{y})+h^pf_{p}(\tilde{y})+\cdots,
\end{equation*}
where $f_{p}(y)=-\delta_f(y)$.
\end{theorem}
\begin{proof}
The function $f_k$ is obtained from
\begin{equation*}
\Phi_{h,f_h^k} = \phi_{h,f} - h^{k+1}f_{k}+O(h^{k+2}),
\end{equation*}
which concludes the proof.
\end{proof}

Denote the truncation of series in (\ref{eq:imde}) as
\begin{equation*}
f_h^K(y) = \sum_{k=0}^K h^k f_k(y).
\end{equation*}
The above computation procedure implies that
\begin{equation*}
\Phi_{h,f_h^K}(x)=\phi_{h,f}(x)+ O(h^{K+2}),
\end{equation*}
which defines a numerical method of order $K+1$ for (\ref{eq:ODE}).

\subsection{ODE based neural networks}\label{sec:discovery}
The discovery of dynamics is essentially a process of identifying the unknown vector field $f$ (also known as dynamics)  of a dynamical system (\ref{eq:ODE}) using provided information of the flow map on given phase points (typically are the states at equidistant time steps of a trajectory, and are written as $\{(x_i,\phi_{T,f}(x_i))\}_{i=1}^I$  in this paper for generality). In this paper, we assume the state set is exact.

Below we briefly recall existing data-driven discovery models using neural network, including ODE-nets, linear multistep neural networks and Hamiltonian neural networks.

\subsubsection{ODE-nets}
Recently, neural ODE \cite{chen2018neural} is proposed as a continuous model by embedding a neural network into an ODE solver. Before the introduction of neural ODE, there were multiple pioneering efforts combining neural networks and ODE solver to discovery the hidden dynamics \cite{anderson1996comparison,gonzalez1998identification,rico1994continuous,rico1993continuous}. In the literature, these models are known as ODE-nets. Using such models, the governing function $f$ is approximated by neural networks via optimizing
\begin{equation}\label{eq:odenloss}
\inf_{u \in \Gamma}  \int_{\mathcal{X}} l(\text{ODESolve}(x, u, T), \phi_{T,f}(x))d P(x).
\end{equation}
Here $l(\cdot,\cdot)$ is a loss function that is minimized when its two arguments are equal (a common choice for regression problem is the square loss $l(y,\hat{y})= \|y-\hat{y}\|_2^2$). $P(x)$ is a probability measure on $\mathcal{X}$ modelling the input distribution which is unknown in practice. In the setting of discovery, we sample training data $\{(x_i,\phi_{T,f}(x_i))\}_{i=1}^N$ and set $P(x)$ to be the empirical measure $P = \sum_{i=1}^I I^{-1} \delta_{x_i}$, yielding the empirical risk optimization problem
\begin{equation*}
\inf_{u \in \Gamma} \frac{1}{I}\sum_{i=1}^I l(\text{ODESolve}(x_i, u, T), \phi_{T,f}(x_i)).
\end{equation*}
We denote the obtained neural network as $f_{net}$. The desired purpose is that $f_{net}$ achieves small loss in the unknown data.
Neural network framework generalize well in practice, although its performance has not been complete explained by most existing theoretical works.

\subsubsection{Linear multistep neural networks}

Linear multistep neural networks (LMNets), developed in \cite{raissi2018multistep}, apply linear multistep methods and neural networks to discovery of dynamics provided given state $y$ on a trajectory at equidistant steps. For LMNets, the unknown $f$ is replaced by neural networks in (\ref{eq:LMM}) and is learned by solving the optimization problem
\begin{equation*}
\inf_{u\in \Gamma} \sum_{i=0}^{I-M}\norm{\sum_{m=0}^M h^{-1}\alpha_my_{i+m} - \sum_{m=0}^M\beta_mu(y_{i+m})}_2^2,
\end{equation*}
where $\Gamma$ is the set of neural networks, $y_i = y(ih)$ with $i=0,\cdots, I$ are the given temporal data-snapshots.

\subsubsection{Hamiltonian neural networks}
Although ODE-nets have remarkable abilities to learn and generalize from data, a vast amount of prior knowledge have not been well utilized. Encoding prior information into a learning algorithm has attracted increasing attention recently \cite{jin2020sympnets, lu2021learning,qin2020machine}. In this paper, we will investigate Hamiltonian neural networks (HNN) \cite{bertalan2019learning,greydanus2019hamiltonian}, in which the unknown Hamiltonian function $H$ instead of the total vector field $f$ is parameterized.

The methodology of HNN is to represent the Hamiltonian $H(y)$ by neural network $u$ and compute $J^{-1} \nabla u(y)$ via auto-differentiation. Subsequently, the approximation is obtained within ODE-net framework, i.e., solving the optimization problem
\begin{equation*}
\inf_{u \in \Gamma}  \frac{1}{I} \sum_{i=1}^I\norm{\text{ODESolve}(x_i, J^{-1} \nabla u, T)-\phi_{T,J^{-1} \nabla H}(x_i)}_2^2,
\end{equation*}
where $\Gamma$ is the set of neural networks. There have been many
research work focusing on HNN with symplectic integration \cite{chen2020symplectic,tong2021symplectic,xiong2021nonseparable}, this problem will be documented in detail later.

\section{Inverse modified differential equations}\label{sec:IMDE}

Consider a very idealized assumption: the neural networks produce zero loss for complete data, i.e., \begin{equation*}
\phi_{T,f}(x) = \text{ODESolve}(x, f_{net}, T).
\end{equation*}
Meanwhile, an ODE solver can be regarded as a one-step integrator with discrete step $T$. Following the procedure of modified integrator in subsection \ref{mde}, we derive a perturbed equation,
\begin{equation*}
\frac{d}{dt}\tilde{y}(t)=F_T(\tilde{y}(t))
\end{equation*}
such that formally
\begin{equation*}
\phi_{T,f}(x) = \text{ODESolve}(x, F_{T}, T).
\end{equation*}
Thus it is natural to expect that training an ODE-net returns an approximation of $F_T$. Similar discussion holds for LMNets and HNN.  In this paper, we name the perturbed equation (\ref{eq:imde}) as inverse modified differential equation (IMDE), since it is used for analysis of discovery. We will introduce the IMDE corresponding to the aforementioned learning models in this section and present rigorous analysis in next section.

\subsection{Inverse modified differential equations for ODE-net}
Detailed computation procedure of IMDE for one-step integrator has been presented in subsection \ref{mde}. Recall that the ODE solver is fixed $S$ compositions of a integrator $\Phi_{h}$.
The following theorem indicates that the IMDE of $\Phi_h$ coincides with the IMDE of the ODE solver.
\begin{theorem}\label{the:inmde}
For any fixed composition number $S$ and truncation index $K$, there exist unique h-independent functions $f_k$ for $0 \leq k \leq K$ such that, the numerical solution of
\begin{equation*}
\frac{d}{dt}\tilde{y}=f_h^K(\tilde{y})=\sum_{k=0}^K h^kf_k(\tilde{y}),
\end{equation*}
satisfies
\begin{equation*}
\Phi_{h,f_h^K}(x)=\phi_{h,f}(x)+ O(h^{K+2})
\end{equation*}
and
\begin{equation*}
\brac{\Phi_{h,f_{h}^K}}^S (x) = \phi_{Sh,f}(x)+ O(h^{K+2})
\end{equation*}
for arbitrary initial value $x$.
\end{theorem}
\begin{proof}
The proof can be found in Appendix \ref{sec:inmde}.
\end{proof}

\subsection{Inverse modified differential equations for LMNet}
According to Theorem \ref{the:step-transition operator}, the formal expansion of $\Phi_{h,f_h}(x)$ for a linear multistep method also exisits, thus IMDE computation can be directly applied to step-transition operators. Using Lie derivatives, we introduce a new approach to derive explicit recursion of IMDE directly from the multistep formula (\ref{eq:LMM}).

\begin{theorem}\label{the:inmde_lmnet}
Consider a weakly stable and consistent multistep method (\ref{eq:LMM}), there exist unique h-independent functions $f_k$ for $0 \leq k \leq K$ such that $f_h^K =\sum_{k=0}^K h^kf_k$ satisfies
\begin{equation}\label{eq:multistep}
\sum_{m=0}^M \alpha_m\phi_{mh,f}(x) = h\sum_{m=0}^M\beta_mf_h^K(\phi_{mh,f}(x)) +O(h^{K+2})
\end{equation}
for arbitrary initial value $x$. In particular, for $k \geq 0$, the functions $f_{k}$ are given as
\begin{equation}\label{eq:multistep inmeq}
\begin{aligned}
f_k(y)=&\frac{1}{(\sum_{m=0}^M\beta_m)} \sum_{m=0}^M \alpha_m \frac{m^{k+1}}{(k+1)!}(D^{k}f)(y)\\
&-\frac{1}{(\sum_{m=0}^M\beta_m)}\sum_{m=0}^M\beta_m\sum_{j=1}^k\frac{m^j}{j!}(D^jf_{k-j})(y).
\end{aligned}
\end{equation}
\end{theorem}
Here, the consistency requires
\begin{equation*}
\sum_{m=0}^M \alpha_m =0,\ \sum_{m=0}^M m\cdot \alpha_m =\sum_{m=0}^M\beta_m,
\end{equation*}
which yields $\sum_{m=0}^M\beta_m\neq 0$ due to the weak stability condition (\ref{eq:weak stability}).

\begin{proof}
The proof can be found in Appendix \ref{sec:inmde_lmnet}.
\end{proof}

\subsection{Learning Hamiltonian system and HNN}\label{sec:Hamiltonian neural networks}

For Hamiltonian system
\begin{equation*}
\frac{d}{dt}y = J^{-1} \nabla H(y),
\end{equation*}
applying Theorem \ref{the:inmde} yields a unique IMDE such that formally
\begin{equation*}
\phi_{T,J^{-1} \nabla H}(x)= \text{ODESolve}(x, f_h, T).
\end{equation*}
Therefore, learning Hamiltonian system, or conservation law, requires the IMDE to be a Hamiltonian system, i.e., $J f_h$ is a potential field. This is true when the numerical integrator used in ODE-net is symplectic.
\begin{theorem}\label{the:modiHam}
Consider a Hamiltonian system with a smooth Hamiltonian $H$, if the numerical integrator $\Phi_h$ is symplectic, then its IMDE is also a Hamiltonian system, i.e., there locally exist smooth functions $H_k$, $k=0,1,2\cdots$, such that
\begin{equation*}
f_k(y)=J^{-1}\nabla H_k(y).\end{equation*}
\end{theorem}
\begin{proof}
This statement has been found in \cite{chartier2007numerical}. We provide a complete proof
in Appendix \ref{sec:modiHam}.
\end{proof}

Non-symplectic numerical integrator can not guarantee that its IMDE is always a Hamiltonian system. Thus ODE-net using non-symplectic integrators and LMNet tend not to learn conservation laws. We remark that both linear multistep method and explicit Runge-Kutta method can not be always symplectic \cite{hairer2006geometric,tang1993symplecticity}. This statement was discussed in \cite{greydanus2019hamiltonian}, while IMDE reveal this problem theoretically.


Furthermore, Theorem \ref{the:modiHam} also reveals the behaviour of HNN. It indicates that HNN with symplectic integrator have certain approximation target. On the contrary, using non-symplectic integrators in HNN can lead to excessive loss and uncertain results dominated by data distribution.

\subsection{Discussion on uniqueness}
We consider the differential equation
\begin{equation*}
\begin{aligned}
&\frac{d}{dt}p=a, \\
&\frac{d}{dt}q=\sin{(p+b)},
\end{aligned}
\end{equation*}
with parameters $a, b$ and initial value $(p(0),q(0))=(p_0,q_0)$. The exact solution is given as
\begin{equation*}
\begin{aligned}
&p(t)=p_0+at, \\
&q(t)=q_0-\frac{1}{a}(\cos{(p_0+at+b)}-\cos{(p_0+b)}).
\end{aligned}
\end{equation*}
When $t=\frac{2\pi}{a}$, we have
\begin{equation*}
\begin{aligned}
&p=p_0+2\pi, \\
&q=q_0.
\end{aligned}
\end{equation*}
Thus, same exact solutions are obtained although the parameter $b$ is different.

In addition, consider a linear equation
\begin{equation*}
\frac{d}{dt} p = \lambda p
\end{equation*}
with parameter $\lambda$. Applying explicit Euler method twice yields
\begin{equation*}
p_1 = (1+\lambda h)^2 p_0 = (1+ (-2/h-\lambda)h)^2p_0.
\end{equation*}
Same numerical solutions are obtained for parameter $\lambda$ and $(-2/h-\lambda)$.

The above examples indicate non-uniqueness of the solution even though $f$ is smooth. We need additional assumptions for rigorous analysis. These problems will be discussed in next section.

\section{Error analysis for discovery using ODE-net}\label{sec:error}
To begin with, we introduce some notations. For a compact subset $\mathcal{K} \subset \C^N$, let $\mathcal{B}(x,r) \subset \C^N$ be the complex ball of radius $r>0$ centered at $x\in\C^N$ and define
\begin{equation*}
\mathcal{B}(\mathcal{K}, r) = \bigcup_{x \in \mathcal{K}} \mathcal{B}(x,r).
\end{equation*}
We will work with $l_{\infty}$- norm on $\C^N$, denote $\norm{\cdot} = \norm{\cdot}_{\infty}$, and for a real analytic vector field $f$, define
\begin{equation*}
\norm{f}_{\mathcal{K}} = \sup_{x\in\mathcal{K}}\norm{f(x)}.
\end{equation*}
Now, the main theorem is given as follows.

\begin{theorem}\label{thm:error}
For $x \in \R^N$ and $r_1, r_2 >0$, a given ODE solver that is $S$ compositions of a $p$th-order Runge-Kutta method $\Phi_{h}$, we denote
\begin{equation*}
\mathcal{L} = \norm{ \text{ODESolve}(\cdot, f_{net}, T) -\phi_{Sh,f}(\cdot)}_{\mathcal{B}(x, r_1)}/T,
\end{equation*}
and suppose the target vector field $f$ and the learned vector field $f_{net}$ are real analytic and bounded by $m$ on $\mathcal{B}(x,r_1+r_2)$, i.e.,
\begin{equation}\label{con:funcbound}
\norm{f}_{\mathcal{B}(x,r_1+r_2)}  \leq m, \ \norm{f_{net}}_{\mathcal{B}(x,r_1+r_2)}  \leq m.
\end{equation}
Then, there exist integer $K=K(T)$ and constants $T_0$, $q$, $\gamma$, $c_1$, $c_2$, $C$ that depend on $m$, $r_1$, $r_2$ and the ODE solver, such that, if $0<T<T_0$,
\begin{equation*}
\begin{aligned}
&\norm{f_{net}(x) - f_h^K(x)} \leq c_1 e^{-\gamma/T^{1/q}} + C \mathcal{L},\\
&\norm{f_{net}(x) - f(x)}\leq  c_2h^p + C\mathcal{L},
\end{aligned}
\end{equation*}
where $h=T/S$ and $f_h^K$ is the truncated vector field of IMDE of $\Phi_{h,f}$.
\end{theorem}




Here, the generalization requirement, i.e., using $\mathcal{L}$ as error bound, is in some sense necessary. Otherwise, if the ODE solver is one composition of implicit Euler method, then, there is no information of $f_{net}$ at $x$. The disadvantage is the generalization assumption on complex ball, we conjecture that there is no essential difference between complex and real space. In addition, the analyticity requirement indicates boundness of derivatives of $f$ due to Cauchy's estimate \cite{scheidemann2005introduction}, more precisely,
\begin{equation*}
\norm{f^{(k)}}_{\mathcal{B}(x,r_1)} \leq k!m r_2^{-k}, \quad k\geq 0,
\end{equation*}
which checks off the high-frequency solution and indicates that our results only hold for low-frequency discovery. In classical regression problems, training FNN first captures low-frequency components of the target function and then approximates the high-frequency \cite{luo2019theory,xu2019training}. We conjecture that the implicit regularization is also applied to ODE-net and thus the analyticity assumption of $f_{net}$ holds without any explicit regularization. Numerical results in Section \ref{sec:Numerical results} will validate both facts.

The requirement of Runge-Kutta method is not necessary. For $g$, $\hat{g}$ satisfying $\norm{g}_{\mathcal{B}(\mathcal{K},r)}\leq m$, $\norm{\hat{g}}_{\mathcal{B}(\mathcal{K},r)} \leq m$, we assume the numerical integrator $\Phi_{h}(y)$ satisfies
\begin{equation}\label{con:intbound}
\begin{aligned}
&1. \text{\ Analytic for}\ |h|\leq h_0 = b_{1} r/m\ \text{and}\ y \in \mathcal{K},\\
&2. \norm{\Phi_{h,\hat{g}}-\Phi_{h,g}}_{\mathcal{K}}\leq b_{2}h\norm{\hat{g}-g}_{\mathcal{B}(\mathcal{K},r)}, \ \text{for}\ |h|\leq h_0,\\
&3. \norm{\hat{g}-g}_{\mathcal{K}} \leq \frac{1}{|h|}\norm{\Phi_{h,\hat{g}}(y)-\Phi_{h,g}(y)}_{\mathcal{K}} \\
&\quad + \frac{b_2|h|}{h_1-|h|} \norm{\hat{g}-g}_{\mathcal{B}(\mathcal{K},b_3h_1 m)}, \ \text{for}\ |h|< h_1 \leq h_0.
\end{aligned}
\end{equation}
Here, $b_1,b_2,b_3$ depend only on the method. Regarding an ODE solver as a one-step integrator, once it satisfies condition (\ref{con:intbound}), then Theorem \ref{thm:error} holds.

\section{Numerical results}\label{sec:Numerical results}

In this section, we provide numerical evidences consistent with the theoretical findings. The exact solutions are computed by very high order numerical integrators on very fine mesh. The order of error $E$ with respect to discrete step $h$ is calculated by $\log_2(\frac{E(2h)}{E(h)})$. Several methods have been proposed for training ODE-nets, such as the adjoint method \cite{chen2018neural,li2017maximum} and the auto-differentiation technique \cite{baydin2017automatic}. Since the latter is more stable \cite{gholami2019anode}, we use the straightforward auto-differentiation to optimize MSE (mean squared error) loss without any explicit regularization. To circumvent learning loss, we train the neural network sufficiently and test the results near the dataset \cite{jin2020quantifying,kawaguchi2017generalization}.

We consider two datasets (\romannumeral1) flow data corresponding to  discovery on trajectory and (\romannumeral2) random data on domain corresponding to discovery on domain to verify our statements, respectively. In particular, we will investigate the results obtained by same HNN model on different data domain. \\
\textbf{Flow data}. The training dataset consists of $I+1$ data points
on a single trajectory starting from $x_0$ with shared data step $T$, i.e., $x_0,x_1,\cdots, x_I$ where $x_i = \phi_{iT}(x_0)$. These data points are grouped in pairs before being used in the neural network, and denoted as $\mathcal{T} = \{(x_i, x_{i+1})\}_{i=0}^I$. For this type, we define the error between $g$ and $\hat{g}$ by
\begin{equation*}
E(g,\hat{g}) = \int_{t\in [0, IT]} \norm{g(x(t)) - \hat{g}(x(t))}_{\infty}dt,
\end{equation*}
where $x(t) = \phi_{t}(x_0)$ and compute this error on very fine mesh. \\
\textbf{Random data on domain}. The training dataset consists of grouped
pairs of points randomly sampled from the given domain $\mathcal{X}$ with shared data step $T$, i.e., $\mathcal{T} = \{(x_i, \phi_{T}(x_i)\}_{i=0}^I$ where $x_i \in \mathcal{X}$.  For this type, we define the error between $g$ and $\hat{g}$ by
\begin{equation*}
E(g,\hat{g}) = \int_{x \in \mathcal{X}} \norm{g(x) - \hat{g}(x)}_{\infty} dx,
\end{equation*}
and compute this error by Monte Carlo integration.

\subsection{Pendulum problem}
We consider the mathematical pendulum of the form
\begin{equation*}\label{eq:PD}
\begin{aligned}
&\frac{d}{dt}p=-\frac{g}{l}\sin q, \\
&\frac{d}{dt}q=p.
\end{aligned}
\end{equation*}
\subsubsection{IMDE for ODE-net}
To begin with, we check out the results for ODE-nets. Here, the chosen numerical methods are the first order Euler method
\begin{equation*}
\begin{aligned}
\Phi_h(y) = y+h f(y),
\end{aligned}
\end{equation*}
with the truncation of the IMDE of order 3 given as
\begin{equation*}
\begin{aligned}
f_{h}^3(y)=&f(y) + \frac{h}{2}f'f(y) + \frac{h^2}{6}f''(f,f)(y)+\frac{h^2}{6}f'f'f(y)\\
&+ \frac{h^3}{24}f'''(f,f,f)(y)+\frac{h^3}{8}f''(f'f,f)(y)\\
&+\frac{h^3}{24}f'f''(f,f)(y)+\frac{h^3}{24}f'f'f'f(y);
\end{aligned}
\end{equation*}
the first order implicit Euler method
\begin{equation*}
\begin{aligned}
\Phi_h(y) = y+h f(\Phi_h(y)),
\end{aligned}
\end{equation*}
with the truncation of the IMDE of order 3 given as
\begin{equation*}
\begin{aligned}
f_{h}^3(y)=&f(y) - \frac{h}{2}f'f(y) + \frac{h^2}{6}f''(f,f)(y)+\frac{h^2}{6}f'f'f(y)\\
&- \frac{h^3}{24}f'''(f,f,f)(y)- \frac{h^3}{8}f''(f'f,f)(y)\\
&- \frac{h^3}{24}f'f''(f,f)(y)- \frac{h^3}{24}f'f'f'f(y);
\end{aligned}
\end{equation*}
and the second order explicit midpoint rule
\begin{equation*}
\begin{aligned}
\Phi_h(y) = y+h f(y+\frac{h}{2}f(y)),
\end{aligned}
\end{equation*}
with the truncation of the IMDE of order 3 given as
\begin{equation*}
\begin{aligned}
f_{h}^3(y)=&f(y) + \frac{h^2}{6}f'f'f(y)+\frac{h^2}{24}f''(f,f)(y)\\
&- \frac{h^3}{16} f'f''(f,f)(y)-\frac{h^3}{8}f'f'f'f(y).
\end{aligned}
\end{equation*}
\begin{figure}[htbp]
    \centering
    \includegraphics[width=0.48\textwidth]{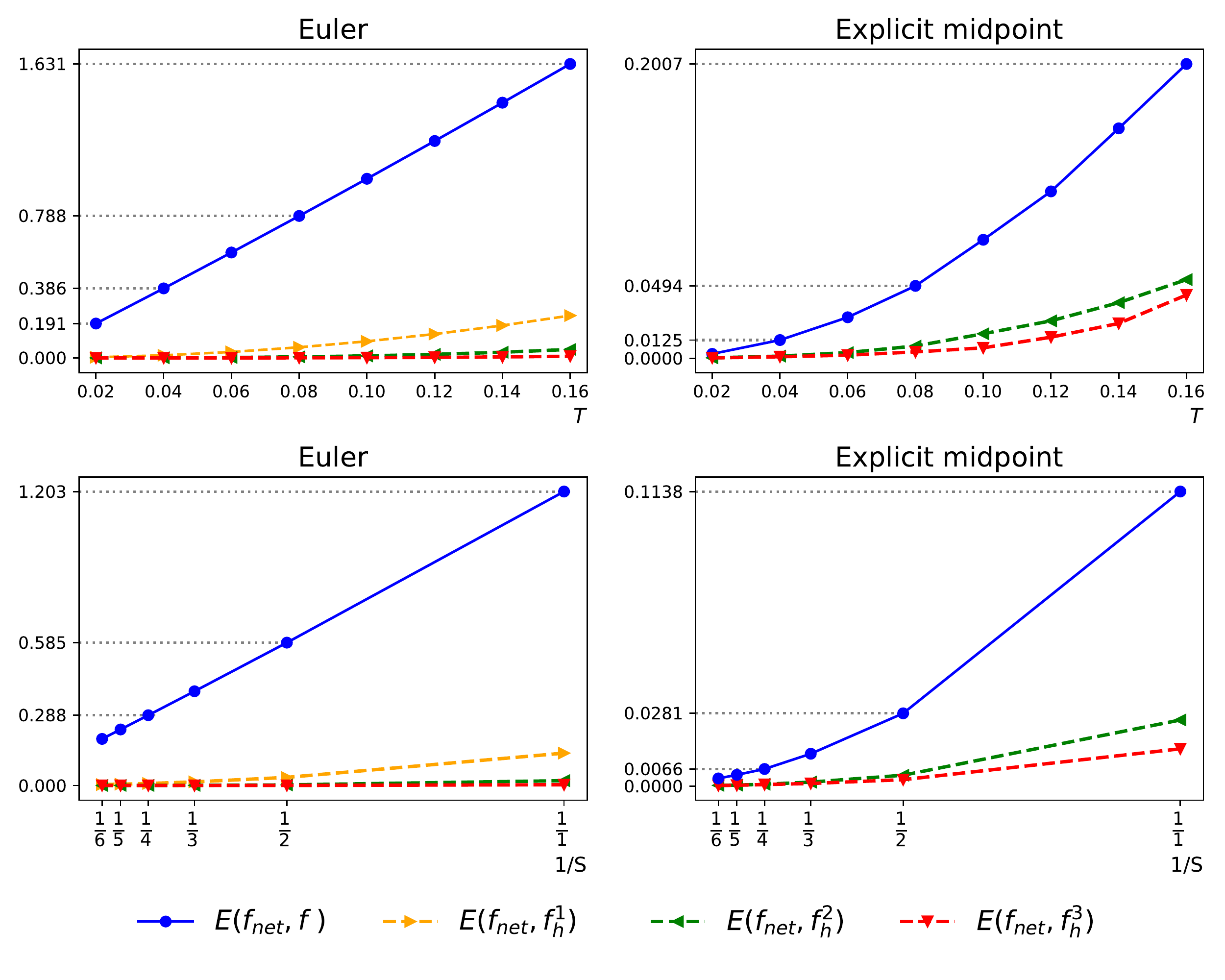}
    \caption{\textbf{Error versus $h$ for pendulum problem using flow data.} (\textbf{Top row}) Composition number $S$ is fixed to 1 thus $h=T$. (\textbf{Bottom row}) Data step $T$ is fixed to $0.12$ thus $h=0.12/S$.}
    \label{fig:PD_error}
\end{figure}

We set $m=l=1,g=10$. Neural networks employed are all two hidden layer and 128 neurons. The activation function is chosen to be tanh. We use Adam optimization \cite{kingma2014adam} where the learning rate is set to decay exponentially with linearly decreasing powers from $10^{-2}$ to $10^{-5}$. Results are collected after $3 \times 10^{5}$ parameter updates in ODE-net framework for Euler and explicit midpoint methods and in LMNet framework for implicit Euler.

We first sample flow data on a single trajectory from $t=0$ to $t=4$ with data step $T$ and initial condition $y_0=(0,1)$. After training, we record the error in Fig. \ref{fig:PD_error} top for different data step $T$ (with one composition thus $h=T$) and bottom for different composition number $S$ (with $T=0.12$ thus $h=0.12/S$). The error between $f$ and trained $f_{net}$ with respect to $h$ increase linearly for Euler while superlinearly for explicit midpoint, more precisely, the convergence order is 1.03 for Euler method while 2.02 for explicit midpoint rule.

\begin{figure}[htbp]
    \centering
    \includegraphics[width=0.48\textwidth]{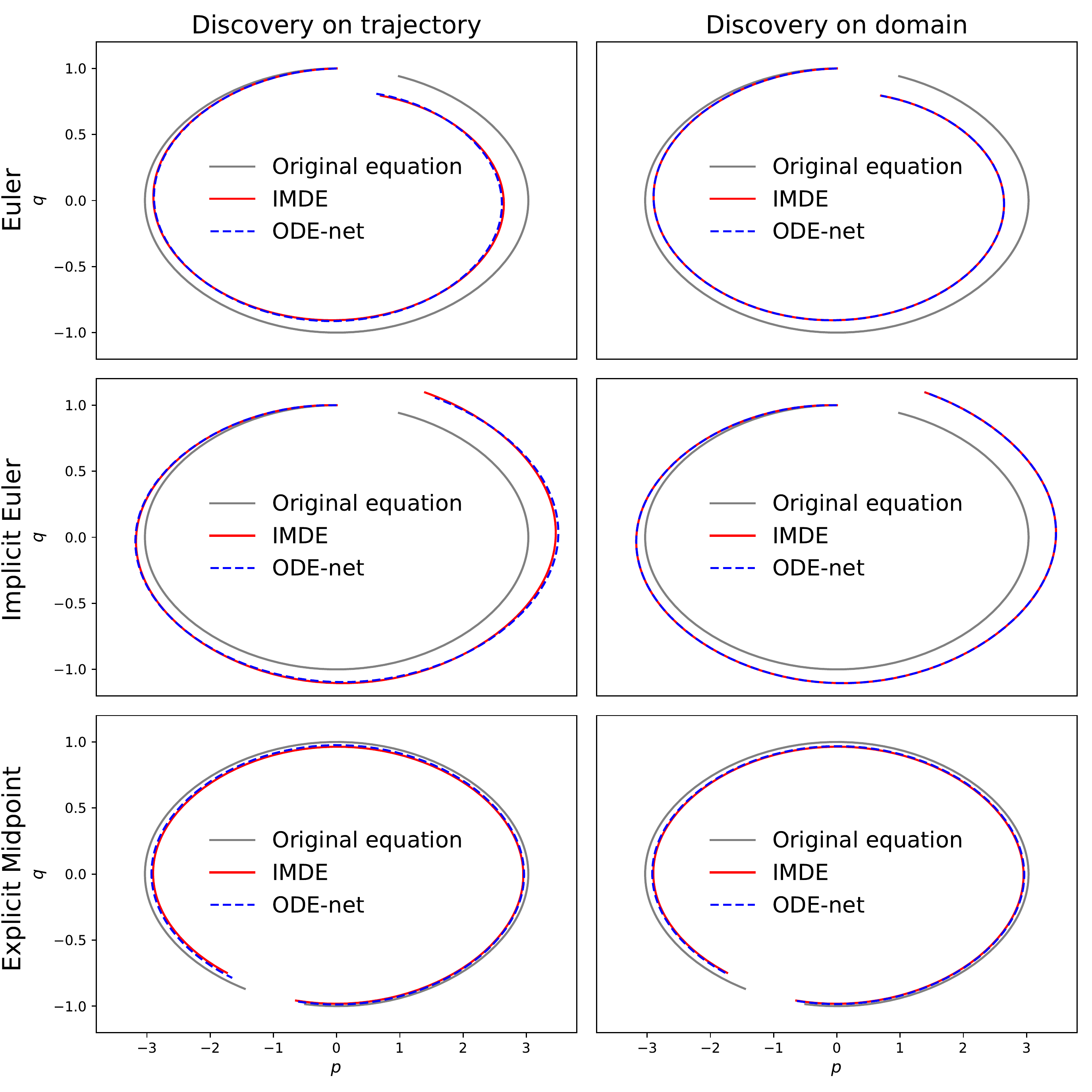}
    \caption{\textbf{Pendulum problem.} Phase portraits starting at $(0,1)$ from $t=0$ to $t=2$ for Euler and Implicit Euler while from $t=1$ to $t=3$ for explicit midpoint. Composition number is fixed to 1 and datastep is $0.02$ for Euler and implicit Euler while $0.12$ for explicit midpoint rule.}
    \label{fig:PD_flow}
\end{figure}

Meanwhile, in Fig. \ref{fig:PD_error}, the error markedly decreases with the increasing of the truncation order. We also depict the orbits starting at $(0,1)$ on the left column of Fig. \ref{fig:PD_flow}, where the learned dynamical systems capture the evolution of the corresponding IMDE. The right column of Fig. \ref{fig:PD_flow} show the performance when the data is randomly sampled from space $[-3.8, 3.8]\times[-1.2, 1.2]$. Here, the learned dynamical systems approximate the corresponding IMDE more accurately since sufficient data leads to better generalization. These results indicate that training ODE-net returns approximations of the IMDE, which is consistent with the theoretical findings.

\subsubsection{IMDE for HNN}
\begin{figure*}[htbp]
    \centering
    \includegraphics[width=0.96\textwidth]{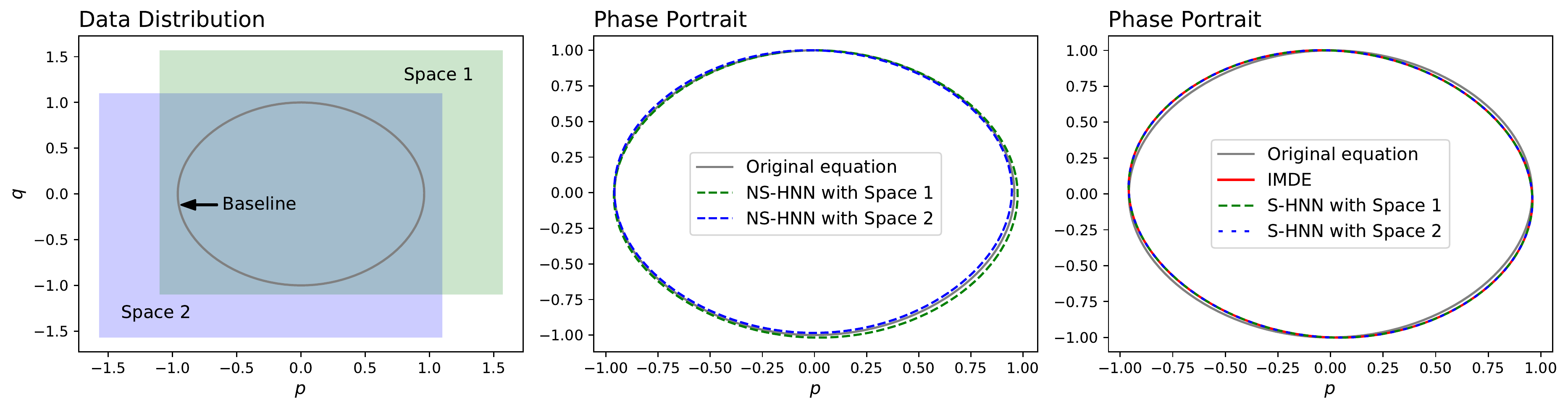}
    \caption{\textbf{Pendulum system.} (\textbf{Left}) Data distribution. (\textbf{Middle}) phase portrait for HNN with non-symplectic integrators. (\textbf{Right}) phase portrait for HNN with symplectic integrators.}
    \label{fig:pendulum}
\end{figure*}
The pendulum is a Hamiltonian system having the Hamiltonian
\begin{equation*}
H(p,q)=\frac{1}{2}p^2 - \frac{g}{l}\cos q,
\end{equation*}
and we also verify the assertion for HNN using this model.  Below we call the HNN with symplectic (non-symplectic) integrator as S-HNN (NS-HNN). Here, we set $l=m=g=1$, $T = 0.1$ and randomly sample training data with number 6000 from Space 1, $[-1.1,\frac{\pi}{2}]\times [-1.1,\frac{\pi}{2}]$, or Space 2, $[-\frac{\pi}{2}, 1.1]\times [-\frac{\pi}{2}, 1.1]$. This data distribution is plotted on the left of Fig. \ref{fig:pendulum}. Test data is generated in the same way with number 100. Neural network architecture employed in HNN is the same as above. Results are collected after $5 \times 10^5$ parameter updates by using Adam optimization with learning rate $1\times 10^{-3}$.  The chosen integrator is the explicit Euler method for NS-HNN and the symplectic Euler method for S-HNN. The symplectic Euler method is given by
\begin{equation*}
\begin{aligned}
\bar{p}=& p - h \frac{\partial H(\bar{p},q)}{\partial q},\\
\bar{q}=&q + h \frac{\partial H(\bar{p},q)}{\partial p},
\end{aligned}
\end{equation*}
which is symplectic and of order 1, $\Phi_h(p,q)=(\bar{p},\bar{q})$. The truncation of the IMDE of order 2 is a Hamiltonian system, and the Hamiltonian is
\begin{footnotesize}
\begin{equation*}
\begin{aligned}
H_{h}^2(p,q)=&H(p,q) + \frac{h}{2}\frac{\partial H}{\partial p} \frac{\partial H}{\partial q}(p,q)+ \frac{h^2}{6}\frac{\partial^2 H}{\partial p^2} (\frac{\partial H}{\partial q},\frac{\partial H}{\partial q})(p,q)  \\
&+ \frac{h^2}{6}\frac{\partial^2 H}{\partial p \partial q} (\frac{\partial H}{\partial p},\frac{\partial H}{\partial q})(p,q) +\frac{h^2}{6} \frac{\partial^2 H}{\partial q^2} (\frac{\partial H}{\partial p},\frac{\partial H}{\partial p})(p,q).
\end{aligned}
\end{equation*}
\end{footnotesize}
Since symplectic integrator is implicit in general, we train it like LMNet, i.e., optimizing
\begin{equation*}
\frac{1}{I}\sum_{i=1}^I \norm{p_i - h \frac{\partial u(\bar{p}_i,q_i)}{\partial q} -\bar{p}_i}_2^2+\norm{
q_i + h \frac{\partial u(\bar{p}_i,q_i)}{\partial p}-\bar{q}_i}_2^2,
\end{equation*}
where $(\bar{p}_i, \bar{q}_i) = \phi_h(p_i,q_i)$ and $u$ is neural network.

\begin{table}[htbp]
    \centering
    \begin{tabular}{lccc}
    \hline
         Integrator     & Space & Training loss & Test loss  \\
     \hline
        Explicit Euler  & 1     &$8.19\times10^{-6 }$  & $8.09\times10^{-6 } $   \\

        Explicit Euler  & 2     &$8.18\times10^{-6 }$  & $7.98\times10^{-6 } $ \\

        Symplectic Euler& 1     &$1.39\times10^{-10}$ & $1.68\times10^{-10}$ \\

        Symplectic Euler& 2     &$1.38\times10^{-10}$ & $1.25\times10^{-10}$ \\
     \hline
    \end{tabular}
    \caption{Training loss and test loss of HNN.}
    \label{tab:pendulumloss}
\end{table}

After training, we solve the exact solutions using initial condition $y_0=(0,1)$ in one period. Fig. \ref{fig:pendulum} shows the exact dynamics of original equation, IMDE and the equations learned by HNN. S-HNN with space 1 and 2 reproduce the phase flow of the same IMDE despite different spaces, while NS-HNN with different data yield discrepant results. Table \ref{tab:pendulumloss} shows the training loss and test loss of HNN. S-HNN achieves lower loss. Clearly, the numerical results support the assertion.

\subsection{Damped oscillator problem}\label{sec:do}
In addition, we consider the two-dimensional damped harmonic oscillator with cubic dynamics, which is also investigated in \cite{keller2021discovery, raissi2018multistep}. The equation is of the form
\begin{equation*}\label{eq:dho}
\begin{aligned}
&\frac{d}{dt}p = -0.1p^3+2.0q^3,\\
&\frac{d}{dt}q = -2.0p^3-0.1q^3.
\end{aligned}
\end{equation*}

Training data is $\mathcal{T}=\{(y_{i},\phi_T(y_{i}))\}_{i=1}^{10000}$, where $y_i=(p_i,q_i)$ are randomly collected from compact set $[-2.2,2.2]\times[-2.2,2]$, $\phi_T(y)$ is the exact solution and $T$ is the data step. Meanwhile, test data is generated in the same way with number of 100. Neural network employed is of one hidden layer and 128 neurons with sigmoid activation. We use batch size of 2000 data points and Adam optimization with learning rate = $1 \times 10^{-4}$. Results are collected after $5 \times 10^{5}$ parameter updates.

\begin{figure}[htbp]
    \centering
    \includegraphics[width=0.48\textwidth]{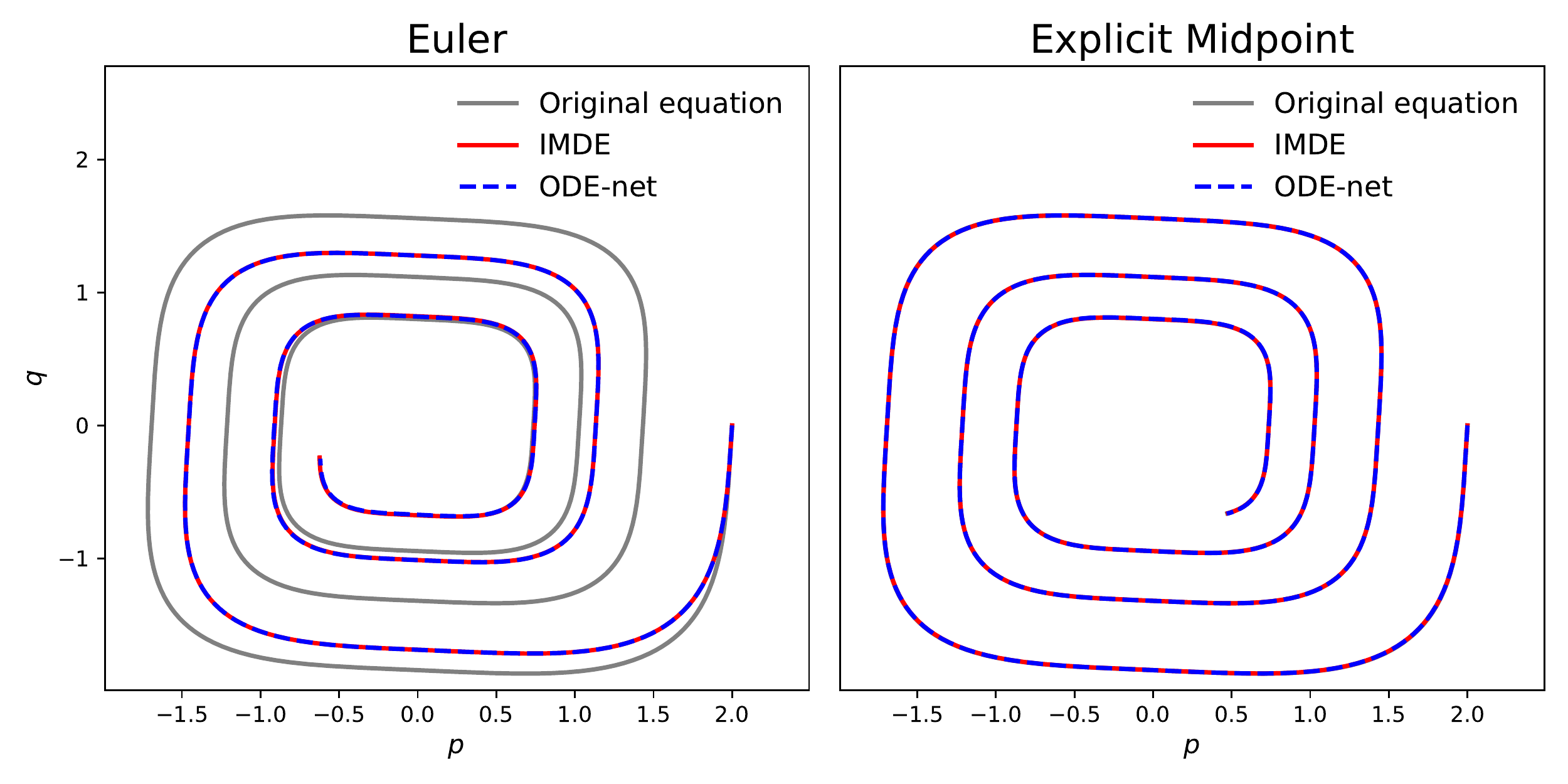}
    \caption{\textbf{Damped harmonic oscillator.} (\textbf{Left}) The ODE solver is two compositions of Euler method. (\textbf{Right}) The ODE solver is two compositions of explicit midpoint rule.}
    \label{fig:DO}
\end{figure}

\begin{table*}[htbp]
    \centering
    \setlength{\tabcolsep}{7pt}{
    \begin{tabular}{ccccccccccc}
    \toprule
    \multirow{2}{*}{DS}&\multirow{2}{*}{CN}&\multicolumn{5}{c}{Damped oscillator with Euler method}&\multicolumn{4}{c}{Lorenz system with explicit midpoint rule}\cr
    \cmidrule(lr){3-7}\cmidrule(lr){8-11}
      &&Training loss& Test loss & E($f_{net}$,$f_{h}^3$) & E($f_{net}$,$f$) & Order & Training loss & E($f_{net}$,$f_{h}^3$) & E($f_{net}$,$f$) & Order\\
     \cmidrule(lr){1-2} \cmidrule(lr){3-7}\cmidrule(lr){8-11}

       0.01  &2 &$1.28\times10^{-9}$ & $1.26\times10^{-9}$   &$3.69\times10^{-3}$& $0.139$& ---&$1.47\times10^{-9}$ & $4.79\times10^{-3}$& $7.87\times10^{-3}$& ---\\

        0.02 &2 &$5.22\times10^{-9}$ & $3.41\times10^{-9}$  & $3.78\times10^{-3}$& $0.277$& 0.992&$4.78\times10^{-9}$ & $4.62\times10^{-3}$& $2.30\times10^{-2}$& 1.55\\

        0.04 &2 &$2.31\times10^{-8}$ & $1.74\times10^{-8}$  & $5.48\times10^{-3}$& $0.547$& 0.982  &$1.35\times10^{-8}$ & $1.16\times10^{-2}$& $8.53\times10^{-2}$& 1.89\\

        0.08 &2 &$1.38\times10^{-7}$ & $1.43\times10^{-7}$  & $3.70\times10^{-2}$& $1.054 $& 0.947 &$4.60\times10^{-7}$ & $1.18\times10^{-1}$& $3.08\times10^{-1}$& 1.85\\
     \cmidrule(lr){1-2} \cmidrule(lr){3-7}\cmidrule(lr){8-11}
        0.04 &8 &$7.60\times10^{-9}$ & $6.55\times10^{-9}$  & $2.95\times10^{-3}$& $0.139$&--- &$1.31\times10^{-8}$ & $4.71\times10^{-3}$& $7.81\times10^{-3}$& ---\\

        0.04 &4 &$9.60\times10^{-9}$ & $1.22\times10^{-8}$  & $3.28\times10^{-3}$& $0.276$& 0.993  &$1.14\times10^{-8}$ & $4.60\times10^{-3}$& $2.29\times10^{-2}$& 1.55\\

        0.04 &2 &$2.31\times10^{-8}$ & $1.74\times10^{-8}$  & $5.48\times10^{-3}$& $0.547$& 0.987 &$1.35\times10^{-8}$ & $1.16\times10^{-2}$& $8.53\times10^{-2}$& 1.89\\

        0.04 &1 &$5.39\times10^{-8}$ & $1.37\times10^{-7}$   & $3.71\times10^{-2}$& $1.056$& 0.949 &$1.61\times10^{-8}$ & $1.16\times10^{-1}$& $3.06\times10^{-1}$& 1.84\\
     \toprule
    \end{tabular}}
    \caption{\textbf{Quantitative results.} DS and IN stand for data step $T$ and composition number $S$, respectively.
    }
    \label{tab:quantitative}
\end{table*}

After training, we solve the exact solutions from $t=0$ to $t=10$ using initial condition $y_0=(2,0)$. Fig. \ref{fig:DO} shows the exact dynamics of original equation, IMDE and the equations learned by ODE-nets. Here, the data step is 0.04. The ODE-net accurately capture the evolution of corresponding IMDE. Note that the original equation and the IMDE on the second row coincide due to the high order integrator.

The quantitative results for Euler method are recorded in Table \ref{tab:quantitative} left side. Here, E($\cdot$,$\cdot$) is calculated by sampling $1\times10^{6}$ points from  $[-2.2,2.2]\times[-2.2,2]$. E($f_{net}$,$f_{h}^3$) is much less than E($f_{net}$,$f$), which again indicates the approximation target is the IMDE. In addition, the order of E($f_{net}$,$f$) with respect to discrete step is approximately 1,
coinciding with Theorem \ref{thm:error}.

\subsection{Lorenz system}\label{sec:ls}
Subsequently, consider the nonlinear Lorenz system
\begin{equation*}
\begin{aligned}
\frac{d}{dt}p =&10(q-p),\\
\frac{d}{dt}q =&p(28-10r)-q,\\
\frac{d}{dt}r =&10pq-\frac{8}{3}r,
\end{aligned}
\end{equation*}
where $y=(p,q,r)$. The training data consists of data points on a single trajectory from $t=0$ to $t=10$ with data step $T$ and initial condition $y_0=(-0.8,0.7,2.6)$. The chosen model architecture and hyper-parameters are the same as in subsection \ref{sec:do} except batch size is 500.
\begin{figure}[htbp]
    \centering
    \includegraphics[width=0.48\textwidth]{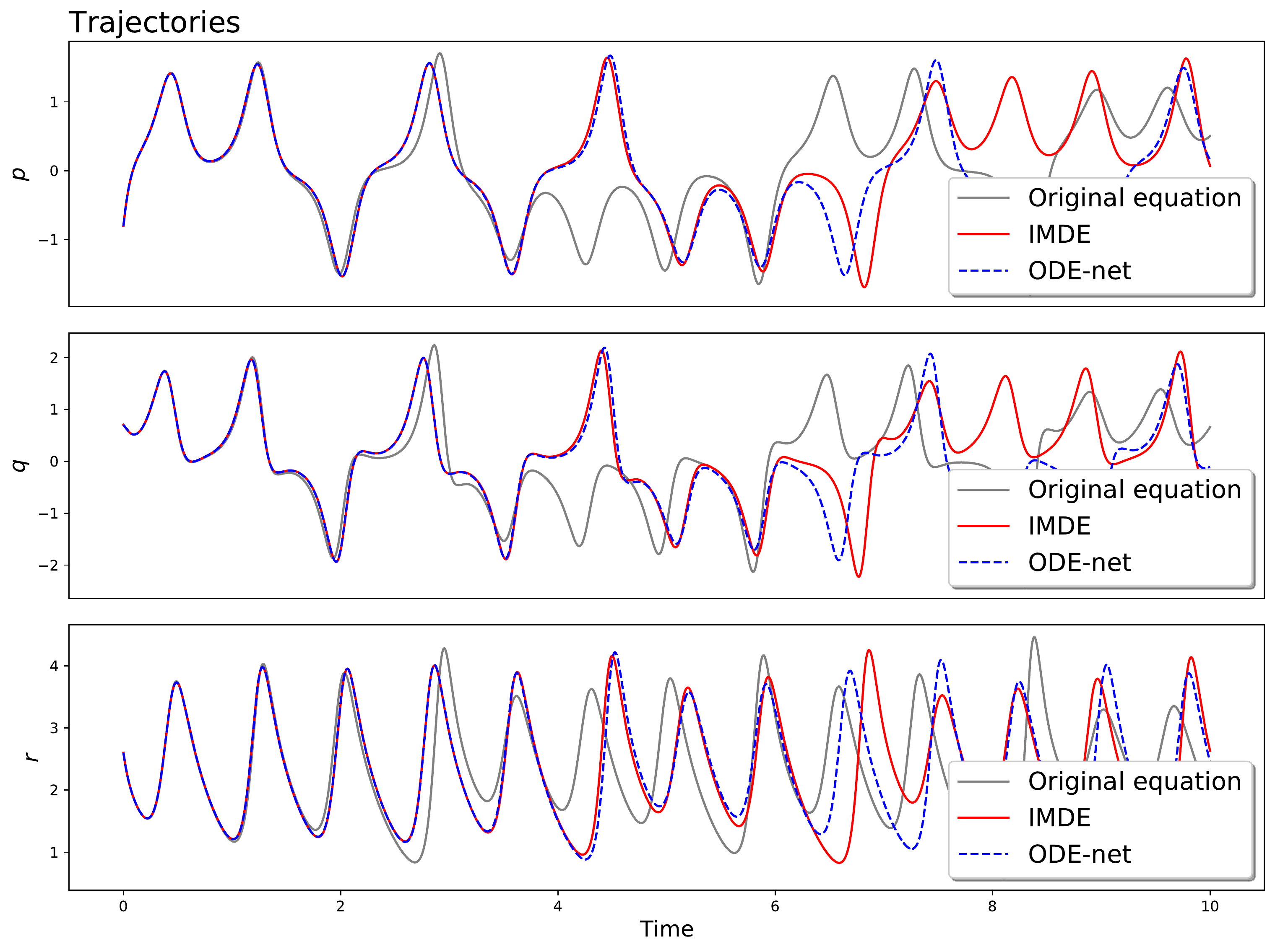}
    \caption{\textbf{Lorenz system.} The trajectories of original equation also represent the training data.}
    \label{fig:LS}
\end{figure}

Upon training, we solve the exact solution from $t=0$ to $t=10$ using initial condition $y_0=(-0.8,0.7,2.6)$. Fig. \ref{fig:LS} depicts the exact trajectories of original equation, IMDE and the equation learned by ODE-net. Here, the data step is 0.04 and the ODE solver is two compositions of explicit midpoint rule. These results could be illuminated by the theoretical findings of this paper. To begin with, the identified system accurately reproduces the trajectories of the IMDE from $t=0$ to $t=4$ due to the generalization ability of neural networks. Then, the ODE-net tries to capture the dynamics of the IMDE, however, there are no sufficient information to tell how $\phi_T$ acts later. Thus the discrepancies explode over time. As demonstrated in Fig. \ref{fig:LS}, the trajectories of the IMDE significantly deviate from the original equation at around $t=4$. Consequently, the identified system drifts away after $t=4$ because of the accumulated errors.

The quantitative results are recorded in Table \ref{tab:quantitative} right side. Here, E($f_{net}$,$f_{h}^3$) is less than E($f_{net}$,$f$) and the order of E($f_{net}$,$f$) with respect to discrete step is approximately 2, which are consistent with the theoretical findings.

\section{Summary}\label{sec:summary}

In this paper, we perform the numerical analysis of discovery of dynamics using ODE based models. The main result is that training an ODE-net returns an approximation of the inverse modified differential equation (IMDE). In addition, the convergence analysis of data-driven discovery using ODE-net is presented, which indicates that the error between trained network $f_{net}$ and the unknown vector field $f$ is bounded by the sum of discrete error $Ch^p$ and learning loss, where $h$ is the discrete step and $p$ is the order of integrator. We also discuss learning Hamiltonian system, IMDE reveals the potential problems and makes clear the behavior of HNN theoretically. Finally, numerical results support the theoretical analysis.

One limitation of our work is the generalization and analyticity  requirements on complex space. Quantifying the generalization
error and implicit regularization for supervised learning are still open research problems. We would like to further investigate such problems for ODE-net in the future.

Low frequency and fine step are essential for both theory and practice. For discovery of high frequencies dynamics, we are inevitably faced with choosing specific ODE solver employed in ODE-net. One possible approach is filtered integrator or  the Modulated Fourier Expansion \cite{hairer2001long}.

Approximation targets depend on the ODE solver. As HNN needs symplectic integrator, further numerical analysis is needed. It is another interesting problem.


\appendices
\section{Proof of Theorem \ref{the:inmde}}\label{sec:inmde}
\begin{proof}
The computation procedure of $f_h$ uniquely defines the functions $f_k$ and can be rewritten as the following recursion:
\begin{equation}\label{recursion}
f_k = \lim_{h\rightarrow 0} \frac{\phi_{h,f} - \Phi_{h,f_h^{k-1}}}{h^{k+1}}.
\end{equation}
We first prove
\begin{equation}\label{Srecursion}
f_k = \lim_{h\rightarrow 0} \frac{\phi_{Sh,f} - \brac{\Phi_{h,f_{h}^{k-1}}}^S}{Sh^{k+1}},
\end{equation}
by induction on $S \geq 1$. First, the case when $S=1$ is obvious. Suppose now that the statement holds for $S-1$. Then, by this inductive hypothesis, we obtain
\begin{equation*}
\begin{aligned}
&\phi_{Sh,f} - \brac{\Phi_{h,f_{h}^{k-1}}}^S \\
=& \Phi_{h,f_{h}^{k-1}} \circ\brac{\phi_{(S-1)h,f} - \brac{\Phi_{h,f_{h}^{k-1}}}^{S-1}}\\
&+ \brac{\phi_{h,f} - \Phi_{h,f_{h}^{k-1}}}\circ \phi_{(S-1)h,f} \\
=& (S-1)h^{k+1} \Phi_{h,f_{h}^{k-1}} \circ f_k + h^{k+1}f_k \circ \phi_{(S-1)h,f} + O(h^{k+2})\\
=&Sh^{k+1}f_k + O(h^{k+2}),
\end{aligned}
\end{equation*}
where we have used the fact that \begin{equation*}
\phi_{(S-1)h,f} = I_N + O\left((S-1)h\right),\ \Phi_{h,f_{h}^{k-1}} = I_N + O(h).
\end{equation*}
Hence the induction is completed.

Suppose that the vector field of IMDE for $\brac{\Phi_{h}}^S$ is of the form $F_{Sh}(y) = \sum_{k=0}^{\infty} (Sh)^k F_k(y)$.
We next prove that $S^kF_k = f_k$ by induction on $k$. First the case when $k=0$ is obvious since $F_0 = f_0=f$. Suppose now $S^kF_k = f_k$ holds for $k\leq K-1$. This inductive hypothesis implies that $F_{Sh}^{K-1}=f_h^{K-1}$. Using (\ref{recursion}) for $F_K$ we obtain
\begin{equation*}
F_K = \lim_{h\rightarrow 0} \frac{\phi_{Sh,f} - \brac{\Phi_{h,F_{Sh}^{K-1}}}^S}{(Sh)^{K+1}}.
\end{equation*}
This together with (\ref{Srecursion}) concludes the induction and thus completes the proof.
\end{proof}

\section{Proof of Theorem \ref{the:inmde_lmnet}}\label{sec:inmde_lmnet}
\begin{proof}
The approach for computation of $f_h$ is presented in two steps. To begin with, by setting $t=mh$ and $F(y)=I_N(y)$ in the formula (\ref{eq:Ld}), the left of (\ref{eq:multistep}) can be expanded as

\begin{align}
\sum_{m=0}^M \alpha_m\phi_{mh,f}(x) =& \sum_{m=0}^M \alpha_m \sum_{k=0}^{\infty}\frac{(mh)^k}{k!}(D^kI_N)(x) \notag \\
=& \sum_{k=0}^{\infty}h^k[\sum_{m=0}^M \alpha_m \frac{m^k}{k!}(D^kI_N)(x)].\label{eq:msleft}
\end{align}
In addition, using (\ref{eq:Ld}) with setting $t=mh$ and $F(y)=f_h(y)$ implies
\begin{small}
\begin{equation*}
\begin{aligned}
h\sum_{m=0}^M\beta_mf_h(\phi_{mh,f}(x))
=&h\sum_{m=0}^M\beta_m \sum_{j=0}^{\infty}\frac{(mh)^j}{j!} \sum_{i=0}^{\infty}h^i(D^jf_i)(x).
\end{aligned}
\end{equation*}
\end{small}
By interchanging the summation order, we obtain
\begin{align}
&h\sum_{m=0}^M\beta_mf_h(\phi_{mh,f}(x))\notag\\
=&h\sum_{m=0}^M\beta_m \sum_{k=0}^{\infty}h^k \sum_{j=0}^k\frac{m^j}{j!}(D^jf_{k-j})(x) \notag\\
=&\sum_{k=0}^{\infty}h^{k+1} \sum_{m=0}^M\beta_m[f_k(x)+\sum_{j=1}^k\frac{m^j}{j!}(D^jf_{k-j})(x)].\label{eq:msright}
\end{align}
Comparing coefficients of $h^k$ in (\ref{eq:msleft}) and (\ref{eq:msright}) for $k=0,1,2,\cdots$ yields
\begin{equation*}
\sum_{m=0}^M\alpha_m=0,
\end{equation*}
the consistency condition, and
\begin{equation*}
\begin{aligned}
&\sum_{m=0}^M\beta_m[f_k(x)+\sum_{j=1}^k\frac{m^j}{j!}(D^jf_{k-j})(x)]\\
=&\sum_{m=0}^M \alpha_m \frac{m^{k+1}}{(k+1)!}(D^{k+1}I_N)(x).
\end{aligned}
\end{equation*}
By plugging $(D^{k+1}I_N)(x)= (D^{k}f)(x)$ and setting $y:=x$, unique $f_k$ are obtained recursively, i.e.,
\begin{equation*}
\begin{aligned}
f_k(y)=&\frac{1}{(\sum_{m=0}^M\beta_m)} \sum_{m=0}^M \alpha_m\frac{m^{k+1}}{(k+1)!}(D^{k}f)(y)\\
&-\frac{1}{(\sum_{m=0}^M\beta_m)}\sum_{m=0}^M\beta_m\sum_{j=1}^k\frac{m^j}{j!}(D^jf_{k-j})(y).
\end{aligned}
\end{equation*}
Here, the right expression only involves $f_j$ with $j< k$ and
\begin{equation*}
\sum_{m=0}^M\beta_m =\sum_{m=0}^M m \alpha_m  \neq 0
\end{equation*}
for weakly stable and consistent methods.
\end{proof}

\section{proof of Theorem \ref{the:modiHam}}\label{sec:modiHam}

\begin{proof}
For a Hamiltonian system (\ref{eq:Hami}), the target function $f$ obeys $f(y)=J^{-1}\nabla H(y)$, which yields $f_0 =J^{-1}\nabla H(y)$. Suppose $f_k(y)=J^{-1}\nabla H_k(y)$ for $k=1,2,\cdots,K$, we need to prove the existence of $H_{K+1}(y)$ satisfying
\begin{equation*}
f_{K+1}(y)=J^{-1}\nabla H_{K+1}(y).
\end{equation*}
By induction, the truncated IMDE
\begin{equation*}
\frac{d}{dt}\tilde{y}=f_{h}^K(\tilde{y})=f(\tilde{y})+hf_1(\tilde{y})+h^2f_2(\tilde{y})+ \cdots +h^{K}f_K(\tilde{y})
\end{equation*}
has the Hamiltonian $H(\tilde{y})+hH_1(\tilde{y})+\cdots+h^{K}H_K(\tilde{y})$. For arbitrary initial value $x$, the numerical solution $\Phi_{h,f_{h}^K}(x)$ satisfies
\begin{equation*}
\phi_{h,f}(x)= \Phi_{h,f_{h}^K}(x)+h^{K+2}f_{K+1}(x)+O(h^{K+3}).
\end{equation*}
And thus
\begin{equation*}
\phi_{h,f}'(x)= \Phi_{h,f_{h}^K}'(x)+h^{K+2}f_{K+1}'(x)+O(h^{K+3}),
\end{equation*}
where $\phi_{h,f}$ and $\Phi_{h,f_{h}^K}$ are symplectic maps, and $\Phi_{h,f_{h}^K}'(y)=I+O(h)$. Then, we have
\begin{equation*}
\begin{aligned}
J=&\phi_{h,f}'(x)^TJ\phi_{h,f}'(x)\\
=&J+h^{K+2}(f_{K+1}'(x)^TJ+Jf_{K+1}'(y))+O(h^{K+3}).
\end{aligned}
\end{equation*}
Consequently, $f_{K+1}'(x)^TJ+Jf_{K+1}'(x)=0$, i.e., $Jf_{K+1}'(x)$ is symmetric. According to the Integrability Lemma \cite[Lemma \uppercase\expandafter{\romannumeral6}.2.7]{hairer2006geometric}, for any $x$, there exists a neighbourhood and a smooth function $H_{K+1}$ obeying
\begin{equation*}
f_{K+1}(x)=J^{-1}\nabla H_{K+1}(x)
\end{equation*}
on this neighbourhood. Hence the induction holds and the proof is completed.
\end{proof}

\section{Proof of Theorem \ref{thm:error}}

\subsection{Properties of Runge-Kutta methods}
To prove Theorem \ref{thm:error}, we firstly prove that the condition (\ref{con:intbound}) is satisfied for Runge-Kutta methods (\ref{runge-kutta}).
\begin{lemma}\label{lem:rk}
For a consistent Runge-Kutta method (\ref{runge-kutta}) denoted as $\Phi_{h}$, let
\begin{equation*}
\mu = \sum_{i=1}^s|b_i|, \quad \kappa =\max_{1\leq i\leq s}\sum_{j=1}^s|a_{ij}|.
\end{equation*}
Consider $g,\hat{g}$ satisfying $\norm{g}_{\mathcal{B}(\mathcal{K},r)}\leq m, \norm{\hat{g}}_{\mathcal{B}(\mathcal{K},r)} \leq m$, if $\kappa \neq0$, then $\Phi_{h,g}$, $\Phi_{h,\hat{g}}$ are analytic for $|h|\leq h_0=r/4\kappa m$ and
\begin{equation*}
\norm{\Phi_{h,\hat{g}}-\Phi_{h,g}}_{\mathcal{K}}\leq 2\mu |h|\norm{\hat{g}-g}_{\mathcal{B}(\mathcal{K},|h|\kappa m)}.
\end{equation*}
Furthermore, for $|h|< h_1 \leq h_0$,
\begin{small}
\begin{equation*}
\norm{\hat{g}-g}_{\mathcal{K}} \leq \frac{\norm{\Phi_{h,\hat{g}}(y)-\Phi_{h,g}(y)}_{\mathcal{K}}}{|h|} +\frac{2\mu|h|\norm{\hat{g}-g}_{\mathcal{B}(\mathcal{K},h_1\kappa m)}}{h_1-|h|} .
\end{equation*}
\end{small}
\end{lemma}
\begin{proof}
For $y \in \mathcal{B}(\mathcal{K}, r/2)$ and $\norm{\Delta y} \leq 1$, the function $\alpha(z) = f(y+ z \Delta y)$ is analytic  for $|z| \leq r/2$ and bounded by $m$. By Cauchy's estimate, we obtain
\begin{equation*}
\norm{f'(y)\Delta y} = \norm{\alpha ' (0)} \leq 2m/r,
\end{equation*}
and $\norm{f'(y)}\leq 2m/r$ for $y \in \mathcal{B}(\mathcal{K}, r/2)$ in the operator norm.

For a Runge-Kutta method (\ref{runge-kutta}) with initial point $y_0\in \mathcal{K}$, the solution can be obtained by the nonlinear systems
\begin{equation*}
\begin{aligned}
&u_i =y_0+h\sum_{j=1}^sa_{ij}\hat{g}(u_j)\quad i=1,\cdots ,s,\\ &\Phi_{h,\hat{g}}(y_0) = y_0+h\sum_{i=1}^sb_i\hat{g}(u_i),\\
&v_i =y_0+h\sum_{j=1}^sa_{ij}g(v_j)\quad i=1,\cdots ,s,\\&\Phi_{h,g}(y_0) = y_0+h\sum_{i=1}^sb_ig(v_i).
\end{aligned}
\end{equation*}
Due to the Implicit Function Theorem \cite{scheidemann2005introduction}, $u_i,v_i$ possess unique solutions on the closed set $\mathcal{B}(\mathcal{K},|h|\kappa m)$ if $2|h|\kappa m/r \leq \gamma<1$ and the method is analytic for $|h|\leq \gamma r/ 2\kappa m$.

In addition,
\begin{equation*}
\begin{aligned}
&\max_{1\leq i\leq s}\norm{u_i-v_i}\\
\leq & |h|\sum_{j=1}^s|a_{ij}|(\norm{\hat{g}(u_j)-\hat{g}(v_j)}+ \norm{\hat{g}(v_j)-g(v_j)})\\
\leq & |h|\kappa \frac{2m}{r}\max_{1\leq j\leq s}\norm{u_j-v_j}+|h|\kappa\norm{\hat{g}-g}_{\mathcal{B}(\mathcal{K},|h|\kappa m)}.
\end{aligned}
\end{equation*}
Thus we obtain
\begin{equation*}
\max_{1\leq i\leq s}\norm{u_i-v_i} \leq \frac{\kappa}{1-|h|\kappa \frac{2m}{r}} |h|\norm{\hat{g}-g}_{\mathcal{B}(\mathcal{K},|h|\kappa m)}.
\end{equation*}
Next, we have
\begin{equation*}
\begin{aligned}
&\norm{\Phi_{h,\hat{g}}(y_0)-\Phi_{h,g}(y_0)}\\
\leq& |h|\sum_{i=1}^s|b_{i}|\norm{\hat{g}(u_i)-\hat{g}(v_i)}+ |h|\sum_{i=1}^s|b_{i}|\norm{\hat{g}(v_i)-g(v_i)}\\
\leq& |h|\mu \frac{2m}{r}\max_{1\leq j\leq s}\norm{u_j-v_j}+|h|\mu \norm{\hat{g}-g}_{\mathcal{B}(\mathcal{K},|h|\kappa m)}\\
\leq& \left(|h|\mu \frac{2m}{r}\frac{\kappa}{1-|h|\kappa \frac{2m}{r}} +\mu \right) |h|\norm{\hat{g}-g}_{\mathcal{B}(\mathcal{K},|h|\kappa m)}.
\end{aligned}
\end{equation*}
Taking $\gamma = 1/2$, together with the arbitrariness of $y_0$, yields
\begin{equation*}
\norm{\Phi_{h,\hat{g}}-\Phi_{h,g}}_{\mathcal{K}} \leq 2\mu |h|\norm{\hat{g}-g}_{\mathcal{B}(\mathcal{K},|h|\kappa m)}.
\end{equation*}
These complete the first part of the proof.

Finally, using Cauchy's estimate, we deduce that for $|h|< h_1 \leq h_0$,
\begin{equation*}
\begin{aligned}
\norm{\frac{d^i}{dh^i}\left(\Phi_{h,\hat{g}}(y_0)-\Phi_{h,g}(y_0)\right)\Big|_{h=0}}
\leq \frac{ i!2\mu  \norm{\hat{g}-g}_{\mathcal{B}(\mathcal{K},h_1\kappa m)}}{h_1^{i-1}}.
\end{aligned}
\end{equation*}
By the analyticity and triangle inequality, we obtain
\begin{small}
\begin{equation*}
\begin{aligned}
&\norm{\Phi_{h,\hat{g}}(y_0)-\Phi_{h,g}(y_0)}\\
\geq& |h|\norm{\hat{g}(y_0)-g(y_0)} - \sum_{i=2}^{\infty}\norm{\frac{h^i}{i!}\frac{d^j}{dh^j}\brac{\Phi_{h,\hat{g}}(y_0)-\Phi_{h,g}(y_0)}\Big|_{h=0}}\\
\geq& |h|\norm{\hat{g}(y_0)-g(y_0)} - 2\mu|h| \norm{\hat{g}-g}_{\mathcal{B}(\mathcal{K},h_1\kappa m)}\sum_{i=2}^{\infty} \left(\frac{|h|}{h_1}\right)^{i-1}. \\
\end{aligned}
\end{equation*}
\end{small}
Therefore,
\begin{small}
\begin{equation*}
\norm{\hat{g}-g}_{\mathcal{K}} \leq \frac{\norm{\Phi_{h,\hat{g}}(y)-\Phi_{h,g}(y)}_{\mathcal{K}}}{|h|} +\frac{2\mu|h|\norm{\hat{g}-g}_{\mathcal{B}(\mathcal{K},h_1\kappa m)}}{h_1-|h|},
\end{equation*}
\end{small}
which concludes the proof.
\end{proof}
We could easily check that for the case $\kappa = 0$, i.e., Euler method, condition (\ref{con:intbound}) also holds.

\subsection{Choice of $K$ and estimation of truncation}
The series in (\ref{eq:imde}) does not converge in general and needs to be truncated. Inspired by the induction idea for conventional modified equations in \cite{reich1999backward}, we prove the truncation estimation for IMDE scenario below.
\begin{lemma}\label{lem:difference}
Let $f(y)$ be analytic in $\mathcal{B}(\mathcal{K}, r)$ and satisfies $\norm{f}_{\mathcal{B}(\mathcal{K}, r)} \leq m$. Suppose the $p$th-order numerical integrator $\Phi_{h}$ satisfies condition (\ref{con:intbound}). Take $\eta = \max\{6, \frac{b_2+1}{29}+1\}$, $\zeta = 10(\eta-1)$, $q = -\ln(2 b_2)/ \ln 0.912$ and $K$ to be the largest integer satisfying
\begin{equation*}
\frac{\zeta(K-p+2)^q|h|m}{\eta r} \leq e^{-q}.
\end{equation*}
If $|h|$ is small enough such that $K\geq p$, then the truncated IMDE satisfies
\begin{equation*}
\begin{aligned}
&\norm{\Phi_{h,f_h^{K}}-\phi_{h,f}}_{\mathcal{K}} \leq b_2\eta m e^{2q-qp}|h|e^{-\gamma /|h|^{1/q}},\\
&\norm{f_{h}^K - f}_{\mathcal{K}}\leq b_2 \eta m \brac{\frac{\zeta m}{b_1r }}^p (1+ 1.38^q d_p) |h|^p,\\
& \norm{f_{h}^K}_{\mathcal{K}} \leq (\eta-1)m,
\end{aligned}
\end{equation*}
where $\gamma = \frac{q}{e}\brac{\frac{b_1r}{\zeta m}}^{1/q}$, $d_p = p^{qp} e^{-q(p-1)}$.
\end{lemma}

\begin{proof}
For $0 \leq \alpha< 1$ and $|h| \leq h_0 = b_1(1-\alpha)r/m$, the condition  (\ref{con:intbound}), together with the fact that $\mathcal{B}(\mathcal{B}(\mathcal{K}, \alpha r), (1-\alpha) r) = \mathcal{B}(\mathcal{K}, r)$ imply
\begin{equation*}
\begin{aligned}
&\norm{\Phi_{h,f}-\phi_{h,f}}_{\mathcal{B}(\mathcal{K}, \alpha r)}\\
\leq& \norm{\Phi_{h,f}-I_N}_{\mathcal{B}(\mathcal{K}, \alpha r)}+ \norm{\phi_{h,f}-I_N}_{\mathcal{B}(\mathcal{K}, \alpha r)}\\
\leq& (b_2+1)|h|m \leq b_1(b_2+1)(1-\alpha)r.
\end{aligned}
\end{equation*}
Here, the map $\Phi_{h,f}-\phi_{h,f}$ contains the factor $h^{p+1}$ since $\Phi_{h,f}$ is of order $p$. By the maximum principle for analytic functions \cite{scheidemann2005introduction}, we obtain
\begin{equation*}
\norm{\frac{\Phi_{h,f}-\phi_{h,f}}{h^{p+1}}}_{\mathcal{B}(\mathcal{K}, \alpha r)} \leq \frac{b_1(b_2+1)(1-\alpha)r}{h_0^{p+1}}
\end{equation*}
and thus (since (\ref{recursion}))
\begin{equation}\label{fp}
\begin{aligned}
\norm{f_p}_{\mathcal{B}(\mathcal{K}, \alpha r)} \leq& \frac{b_1(b_2+1)(1-\alpha)r}{h_0^{p+1}} \\
=& (b_2+1)m\brac{\frac{m}{b_1(1-\alpha)r}}^{p}.
\end{aligned}
\end{equation}
Below we proceed to prove that for $\alpha \in [0,1)$, if
\begin{equation*}
|h|\leq h_k:=\frac{b_1(1-\alpha)r}{\zeta(k-p+1)^qm},
\end{equation*}
then
\begin{equation}\label{fk}
\norm{f_k}_{\mathcal{B}(\mathcal{K}, \alpha r)} \leq b_2\eta  m\brac{\frac{\zeta(k-p+1)^qm}{b_1(1-\alpha)r}}^{k}
\end{equation}
for $k \geq p$ by induction, where $\eta = \max\{6, \frac{b_2+1}{29}+1\}$, $\zeta = 10(\eta-1)$ and $q = -\ln(2 b_2)/ \ln 0.912$. First, the case when $k=p$ is obvious since (\ref{fp}). Suppose now (\ref{fk}) holds for $k \leq K$. If $|h| \leq h_{K+1}$, taking
\begin{equation*}
\delta_{K+1} := \frac{\eta-1}{(K-p+2)^q\zeta} ,\ \beta_K := (1-\delta_{K+1})(K-p+2)^q
\end{equation*}
yields that for $p \leq k \leq K$
\begin{equation*}
|h|\leq \frac{b_1(1-\alpha)r}{\zeta(K-p+2)^qm} \leq \frac{b_1(1-\alpha-\delta_{K+1}(1-\alpha))r}{\zeta(k-p+1)^qm}.
\end{equation*}
Therefore, by inductive hypothesis we obtain
\begin{equation*}
\norm{f_{k}}_{\mathcal{B}(\mathcal{K}, (\alpha+\delta_{K+1} (1-\alpha)) r)} \leq b_2\eta m\brac{\frac{\zeta(k-p+1)^qm}{b_1(1-\delta_{K+1})(1-\alpha)r}}^{k}
\end{equation*}
via replacing $\alpha$ by $\alpha+\delta_{K+1}(1-\alpha) \in [\delta_{K+1}, 1)$ in (\ref{fk}).  This indicates
\begin{equation*}
\begin{aligned}
&\norm{f_{h}^K}_{\mathcal{B}(\mathcal{K}, (\alpha+\delta_{K+1} (1-\alpha)) r)} \\
\leq& m\left[1+ (b_2+1) \brac{\frac{1}{\zeta\beta_K}}^p + b_2 \eta \sum_{k=p+1}^K\brac{\frac{(k-p+1)^q}{\beta_K}}^{k} \right].\\
\end{aligned}
\end{equation*}
Since
\begin{equation*}
\sum_{k=p+1}^K\brac{\frac{(k-p+1)}{(K-p+1.9)}}^{k} \leq 0.912,
\end{equation*}
which is maximal for $K=6$ and $p=1$, and
\begin{equation*}
\sum_{k=p+1}^K\brac{\frac{(k-p+1)^q}{\beta_K}}^{k} \leq \left[\sum_{k=p+1}^K\brac{\frac{(k-p+1)}{(K-p+1.9)}}^{k}\right]^q,
\end{equation*}
we deduce that
\begin{equation}\label{fKbound}
\norm{f_{h}^K}_{\mathcal{B}(\mathcal{K}, (\alpha+\delta_{K+1} (1-\alpha)) r)} \leq (\eta-1)m.
\end{equation}
Here, we have used the definition of $\eta$, $\zeta$ and $q$.
Subsequently, by this estimate and condition (\ref{con:intbound}), we obtain
\begin{equation*}
\begin{aligned}
\norm{\Phi_{h,f_h^K}-I_N}_{\mathcal{B}(\mathcal{K}, \alpha r)} \leq& |h| b_2\norm{f_{h}^K}_{\mathcal{B}(\mathcal{K}, (\alpha+\delta_{K+1} (1-\alpha)) r)} \\
\leq& |h| b_2 (\eta-1)m,
\end{aligned}
\end{equation*}
where
\begin{equation*}
|h|\leq \frac{b_1(1-\alpha)r}{\zeta(K-p+2)^qm} = \frac{b_1\delta_{K+1}(1-\alpha)r}{(\eta-1)m}.
\end{equation*}
And then using triangle inequality yields
\begin{equation*}
\begin{aligned}
\norm{\Phi_{h,f_h^K}-\phi_{h,f}}_{\mathcal{B}(\mathcal{K}, \alpha r)}& \leq h_K b_2 \eta m.\\
\end{aligned}
\end{equation*}
Again by the maximum principle for analytic functions, together with the fact that $\Phi_{h,f_h^K}-\phi_{h,f}$ contains the factor $h^{K+2}$, we deduce that
\begin{equation}\label{phi}
\begin{aligned}
\norm{\frac{\Phi_{h,f_h^K}-\phi_{h,f}}{h^{K+2}}}_{\mathcal{B}(\mathcal{K}, \alpha r)}& \leq  \frac{h_K b_2 \eta m}{h_K^{K+2}}.\\
\end{aligned}
\end{equation}
Again by (\ref{recursion}), we conclude that
\begin{equation*}
\norm{f_{K+1}}_{\mathcal{B}(\mathcal{K}, \alpha r)} \leq  b_2 \eta m\brac{\frac{\zeta (K-p+2)^qm}{b_1(1-\alpha)r}}^{K+1},
\end{equation*}
which completes the induction.

The above induction also shows that (\ref{phi}) holds if $|h| \leq h_{K+1}$. Taking $\alpha = 0$ we have
\begin{equation*}
\norm{\Phi_{h,f_h^K}-\phi_{h,f}}_{\mathcal{K}} \leq  b_2\eta |h|m\brac{ \frac{\zeta(K-p+2)^q|h|m}{b_1r}}^{K+1}.
\end{equation*}
We set $K^*$ to be the largest integer satisfying
\begin{equation*}
\frac{\zeta (K^*-p+2)^q|h|m}{b_1r} \leq e^{-q}.
\end{equation*}
Clearly, $|h| \leq h_{K^*+1}$ with $\alpha=0$. Therefore,
\begin{equation*}
\begin{aligned}
\norm{\Phi_{h,f_h^{K^*}}-\phi_{h,f}}_{\mathcal{K}} \leq& b_2\eta |h|m\brac{ \frac{\zeta(K^*-p+2)^q|h|m}{b_1r}}^{K^*+1}\\
\leq& b_2\eta m e^{2q-qp}|h|e^{-\gamma /|h|^{1/q}},
\end{aligned}
\end{equation*}
where $\gamma = \frac{q}{e}\brac{\frac{b_1r}{\zeta m}}^{1/q}$. The first part of the lemma has been completed.

Next, according to (\ref{fk}) we obtain
\begin{equation*}
\begin{aligned}
&\norm{f_{h}^{K^*} - f}_{\mathcal{K}}\\
\leq& b_2 \eta m \brac{\frac{\zeta |h|m}{b_1r }}^p\Big[1\\
&+\sum_{k=p+1}^{K^*}\frac{(k-p+1)^{qp}}{e^{q(k-p)}}\brac{\frac{(k-p+1)}{({K^*}-p+2)}}^{q(k-p)}\Big]\\
\leq& b_2 \eta m \brac{\frac{\zeta m}{b_1r }}^p (1+ 1.38^q d_p) |h|^p,\\
\end{aligned}
\end{equation*}
where $d_p = p^{qp} e^{-q(p-1)}$ satisfies $d_p \geq (k-p+1)^{qp}e^{-q(k-p)}$ for any $k\geq p+1$.

Finally, we immediately derive the boundedness of $f_h^K$ due to (\ref{fKbound}). The proof has been completed.
\end{proof}

\subsection{Error estimation}

With Lemma \ref{lem:difference}, we first present the error estimation for one-step integrator.

\begin{lemma}\label{lem:error}
For $x \in \R^N$ and $r_1, r_2 >0$, a given  $p$th-order numerical integrator $\Phi_{h}$ satisfying condition (\ref{con:intbound}), we denote
\begin{equation*}
\mathcal{L} = \norm{\Phi_{h,f_{net}}-\phi_{h,f}}_{\mathcal{B}(x, r_1)}/h,
\end{equation*}
and suppose the assumption (\ref{con:funcbound}) are satisfied. Then, there exist integer $K=K(h)$ (as defined in Lemma \ref{lem:difference} with $r=r_2$), and constant $h_0$, $q$, $\gamma$, $c_1$, $c_2$, $C$ that depend on $m$, $r_1$, $r_2$ and the integrator $\Phi_{h}$, such that, if $0<h<h_0$,
\begin{equation*}
\begin{aligned}
&\norm{f_{net}(x) - f_h^K(x)} \leq c_1 e^{-\gamma/h^q} + C \mathcal{L},\\
&\norm{f_{net}(x) - f(x)}\leq  c_2h^p +C \mathcal{L}.\\
\end{aligned}
\end{equation*}
\end{lemma}


\begin{proof}
According to the first inequality of Lemma \ref{lem:difference}, there exist $K, c, \gamma, q$ such that
\begin{equation*}
\norm{\Phi_{h,f_h^{K}}-\phi_{h,f}}_{\mathcal{B}(x, r_1)} \leq ch e^{-\gamma /h^{1/q}},
\end{equation*}
which immediately yields
\begin{equation}\label{Phi fnet-fh}
\delta : = \frac{1}{h}\norm{\Phi_{h, f_{net}}-\Phi_{h,f_h^K}}_{\mathcal{B}(x, r_1)} \leq \mathcal{L} + c e^{-\gamma /h^{1/q}}.
\end{equation}

Next, by the third inequality of Lemma \ref{lem:difference}, there exists $M \geq m$ such that $\norm{f_{h}^K}_{\mathcal{B}(x, r_1)}<M$. Let
\begin{equation*}
h_1 = (eb_2+1)h,  \  \lambda = \frac{b_2h}{h_1-h} = e^{-1}.
\end{equation*}
Using the third item of (\ref{con:intbound}), we have
\begin{equation*}
\begin{aligned}
\norm{f_{net} - f_h^K}_{\mathcal{B}(x, jh_1b_3 M)} \leq \delta + \lambda \norm{f_{net} - f_h^K}_{\mathcal{B}(x, (j+1)h_1b_3 M)}
\end{aligned}
\end{equation*}
for $0\leq j \leq r_1/h_1b_3 M$. This yields
\begin{equation*}
\begin{aligned}
&\norm{f_{net} - f_h^K}_{\mathcal{B}(x, jh_1b_3 M)} - \frac{\delta}{1-\lambda} \\
\leq& \lambda\brac{ \norm{f_{net} - f_h^K}_{\mathcal{B}(x, (j+1)h_1b_3 M)} - \frac{\delta}{1-\lambda}}.
\end{aligned}
\end{equation*}
Therefore,
\begin{equation*}\label{fnet-fh}
\begin{aligned}
\norm{f_{net}(x) - f_h^K(x)} \leq e^{-\hat{\gamma}/h} \norm{f_{net} - f_h^K}_{\mathcal{B}(x, r_1)} + \frac{\delta}{1-\lambda},
\end{aligned}
\end{equation*}
where $\hat{\gamma} = \frac{r_1}{(eb_2+1)b_3M}$. By this estimation and (\ref{Phi fnet-fh}), we conclude that
\begin{equation*}
\norm{f_{net}(x) - f_h^K(x)} \leq c_1 e^{-\gamma /h^{1/q}} + C \mathcal{L},
\end{equation*}
where $C =e/(e-1)$ and $c_1$ is constant satisfying
\begin{equation*}
c_1 \geq C \cdot c+ 2M e^{\gamma/h^{1/q} - \hat{\gamma}/h}.
\end{equation*}

Finally, by the second inequality in Lemma \ref{lem:difference}, we obtain the second estimation and complete the proof.
\end{proof}

With these results, we are able to provide the proof of Theorem \ref{thm:error}.
\begin{proof}[Proof of Theorem \ref{thm:error}]
According to the fact that several compositions of Runge-Kutta methods are again Runge-Kutta methods, Lemma \ref{lem:rk}, Lemma \ref{lem:error} and Theorem \ref{the:inmde}, we conclude the proof.
\end{proof}


%




\ifCLASSOPTIONcaptionsoff
  \newpage
\fi

\bibliographystyle{abbrv}
\bibliography{references}

\end{document}